\documentclass[a4paper,10pt]{article}

\usepackage[british]{babel}
\usepackage{amsmath,amssymb,wasysym,amsfonts,amsthm,amsthm,color}
\usepackage{empheq}

\usepackage{graphicx,epsfig}
\usepackage{subfig}

\usepackage{txfonts}
\usepackage[T1]{fontenc}
\usepackage{textcomp}
\usepackage[utf8x]{inputenc}
\usepackage{ucs}
\usepackage{titletoc,titlesec}

\usepackage[a4paper,margin=1.5in]{geometry}
\usepackage{hyperref,url,enumerate,framed}

\newtheorem{thm}{Theorem}
\newtheorem{prop}[thm]{Proposition}
\newtheorem{lem}[thm]{Lemma}
\newtheorem{cor}[thm]{Corollary}

\newtheorem{defn}[thm]{Definition}
\newtheorem{ass}[thm]{Assumption}

\theoremstyle{remark}
\newtheorem{rmk}[thm]{Remark}
\numberwithin{equation}{section}
\numberwithin{thm}{section}

\newcommand{\eps}{\varepsilon}

\def\ba{{\mathbf{a}}}
\def\bA{{\mathbf{A}}}
\def\bu{{\mathbf{u}}}
\def\bv{{\mathbf{v}}}
\def\bp{{\mathbf{p}}}
\def\bq{{\mathbf{q}}}

\def\bU{{\mathbf{U}}}

\def\bf{{\mathbf{f}}}
\def\bg{{\mathbf{g}}}
\def\bh{{\mathbf{h}}}
\def\bpsi{{\mathbf{\psi}}}
\def\bS{{\mathbf{S}}}
\def\Wus{{\mathcal{W}_s^\mathrm{u}(0,0)}}
\def\Wss{{\mathcal{W}_s^\mathrm{s}(0,0)}}
\def\Wuss{{\mathcal{W}_s^\mathrm{u,s}(0,0)}}
\def\To{T_\mathrm{o}}
\def\ssn{{s_\mathrm{sn}}}
\def\rPCB{{\rho_\mathrm{PCB}}}
\newcommand{\bF}{\mathrm F}
\newcommand{\bG}{\mathrm G}
\newcommand{\bH}{\mathrm H}
\def\bP{{\mathrm{P}}}
\def\obP{\overline{\mathrm{P}}}
\newcommand{\bR}{\mathrm R}
\newcommand{\tPi}{\widetilde{\Pi}}
\newcommand{\bpm}{\begin{pmatrix}}
\newcommand{\epm}{\end{pmatrix}}
\newcommand{\bsub}{\begin{subequations}}
\newcommand{\esub}{\end{subequations}}

\newcommand{\beq}{\begin{equation}}
\newcommand{\eeq}{\end{equation}}

\newcommand{\kpc}[1]{{#1}}
\def\R{{\mathbb R}}
\def\C{{\mathbb C}}
\def\mbL{{\mathbb L}}
\def\hmbL{{\hat{\,\mathbb{L}}}}

\def\mbLzo{{(\mathbb L_1)}_{\rm{out},0}}

\def\cmD{{\mathcal D}}
\def\cmM{{\mathcal M}}

\def\cF{{\cal F}}
\def\cG{{\cal G}}
\def\cL{{\cal L}}
\def\cLzi{{\cal L}_{\rm{in},0}}
\def\cLzo{{\cal L}_{\rm{out},0}}
\def\cLoi{{\cal L}_{\rm{in},1}}

\def\cLzo{{\cal L}_{\rm{out},0}}
\def\cM{{\cal M}}
\def\mrd{{\mathrm d}}
\def\mrL{{\mathrm L}}
\def\mrM{{\mathrm M}}
\def\eps{{\varepsilon}}
\def\mcC{{\mathcal C}}

\def\tJ{{\tilde{J}}}
\def\tPsi{\tilde{\Psi}}
\def\oxi{\overline{\xi}}
\def\busg{\mathbf{u}_{\,\Gamma}}

\title{Robust Stability of Multicomponent Membranes: the Role of Glycolipids}
\author{Yuan Chen, Arjen Doelman, Keith Promislow \& Frits Veerman }

\begin{document}
\maketitle

\begin{abstract} 
We present the multicomponent functionalized free energies that characterize the low-energy packings of
amphiphilic molecules within a membrane through a correspondence to connecting orbits within a reduced dynamical 
system.  To each connecting orbits we associate a manifold of low energy membrane-type configurations 
parameterized by a large class of admissible interfaces. The normal coercivity of the manifolds is established
through criteria depending solely on the structure of the associated connecting orbit. We present a class of examples that 
arise naturally from geometric singular perturbation techniques, focusing on a model that characterizes the stabilizing role of 
cholesterol-like glycolipids within phospholipid membranes.
\end{abstract}

\section{Introduction}
 Amphiphilic molecules play a fundamental role in the self-assembly of nanostructured membranes. These include phospholipids, the building 
 blocks of cellular membranes, and synthetic polymers that are finding applications to drug delivery compounds and as active materials for 
 separator membranes in energy conversion devises, \cite{BAR09, GD07, SRM12}.   The scalar functionalized Cahn-Hilliard free energy \kpc{has
 been proposed \cite{GG94, HP11} as a model of} 
 the interaction of a single species of amphiphilic molecule with a solvent, characterizing the density of the amphiphile through a phase function
  $u\in H^2(\Omega)$ via the free energy
 \beq
 \label{e:scalar-FCH}
 \cF_{\rm FCH}[u]:= \int_\Omega \frac12 (\eps^2 \Delta u- W'(u))^2 - \eps^p \left(\eta_1\frac{\eps^2}{2} |\nabla u|^2 + \eta_2 W(u)\right)\, \mrd x.
 \eeq 
Here $W$ is a double well potential with two unequal depth minima at $b_-<b_+$ satisfying $W(b_-)=0> W(b_+).$ The
amphiphilic volume fraction is related to the density $u-b_-$ with the equilibrium state $u=b_-$ corresponding to pure solvent. 
The strength of the lower order functionalization terms are characterized by the value of $p$, generically selected as 1 or 2,
and the values of $\eta_1$ and $\eta_2$. These parameters encode the affinity of the charged elements of the amphiphilic molecule 
for the solvent (called the solvent quality) and the aspect ratio of the amphiphilic molecule, respectively, see \cite{Barnhill15, BAR09, NP19}.

Experimental investigations show that when single-species amphiphilic materials are dispersed in solvent and \kpc{then allowed to self-assemble, 
in a process called casting, the form} a diverse array of molecular-width structures, \cite{DE02, JB04}. The associated bifurcation diagram depends subtly upon both the aspect ratio 
of the amphiphilic molecule and the solvent quality. Molecules with aspect ratio near unity form two-molecule thick 
bilayer membranes familiar from cellular biology. Larger aspect ratio molecules form higher codimensional structures such as 
filaments and micelles and complex networks with triple junctions and end-caps. Within the casting experiments the genesis of this 
structural diversity has been referred to as the onset of 'morphological complexity', \cite{JB03}. Gradient flows of the scalar FCH free 
energy  provide an accurate representation of this bifurcation structure, providing a mechanism for the onset of morphological complexity 
via a transient passage through a pearling instability that leads bilayers to and break into filaments and other 
higher codimension morphologies, \cite{NP19}.  The single species bilayers supported by the scalar FCH free energy are always 
neutrally stable to pearling bifurcations at leading order -- opening the door for lower order terms, including the system parameters $\eta_1$ and $\eta_2$ and the dynamic value of the bulk density of amphiphilic material  to play a decisive role, \cite{NP18}. Indeed, previous work on the scalar 
FCH has shown that the neutral modes of its  bilayer interfaces  are associated either to motion of the underlying interface, termed meander, or to the
short wave-length modulations of their width associated to pearling, \cite{HayrapetyanPromislow.2015}.  In regimes in which interfaces are
stable to the pearling bifurcation, the interfacial motion has been rigorously described through a normal velocity proportional to curvature, with the 
proportionality constant depending upon the difference between the evolving bulk density of amphiphilic materials away from the interface and the 
bilayer bulk-density equilibrium value. Significantly this proportionality can be negative, which is typical in casting experiments in which the bulk
density is high, and leads to a {\it curve  lengthening} motion regularized by surface diffusion, \cite{CP-20}.

In biologically relevant settings, phospholipid membranes are robustly stable to pearling bifurcations, which would generically be toxic to the living 
cell or to the organelle enclosed by the membrane. Significantly phospholipid membranes are never comprised of a single species. Generically
significant amounts of cholesterol or other glycolipids are blended into the phospholipid membranes.  Indeed all eukaryotic plasma membranes 
contain large amounts of cholesterol, often a 1-1 molar mixture of phospholipids and cholesterol \cite{MBoC-02}.  While phospholipids are classic amphiphilic materials with a charged head group and a hydrophobic tail, cholesterol is a shorter, asymmetric molecule with a small, weakly charged 
head and a hydrophobic body. Within a phospholipid membrane cholesterol typically wedges itself in the void space between the amphiphilic phospholipid molecules, see Figure\,\ref{f:cholesterol-schematic}, where it significantly constrains the motion of the lipids.

\begin{figure}[h]
	\centering
	\begin{tabular}{cc}
		\includegraphics[width=0.5\textwidth]{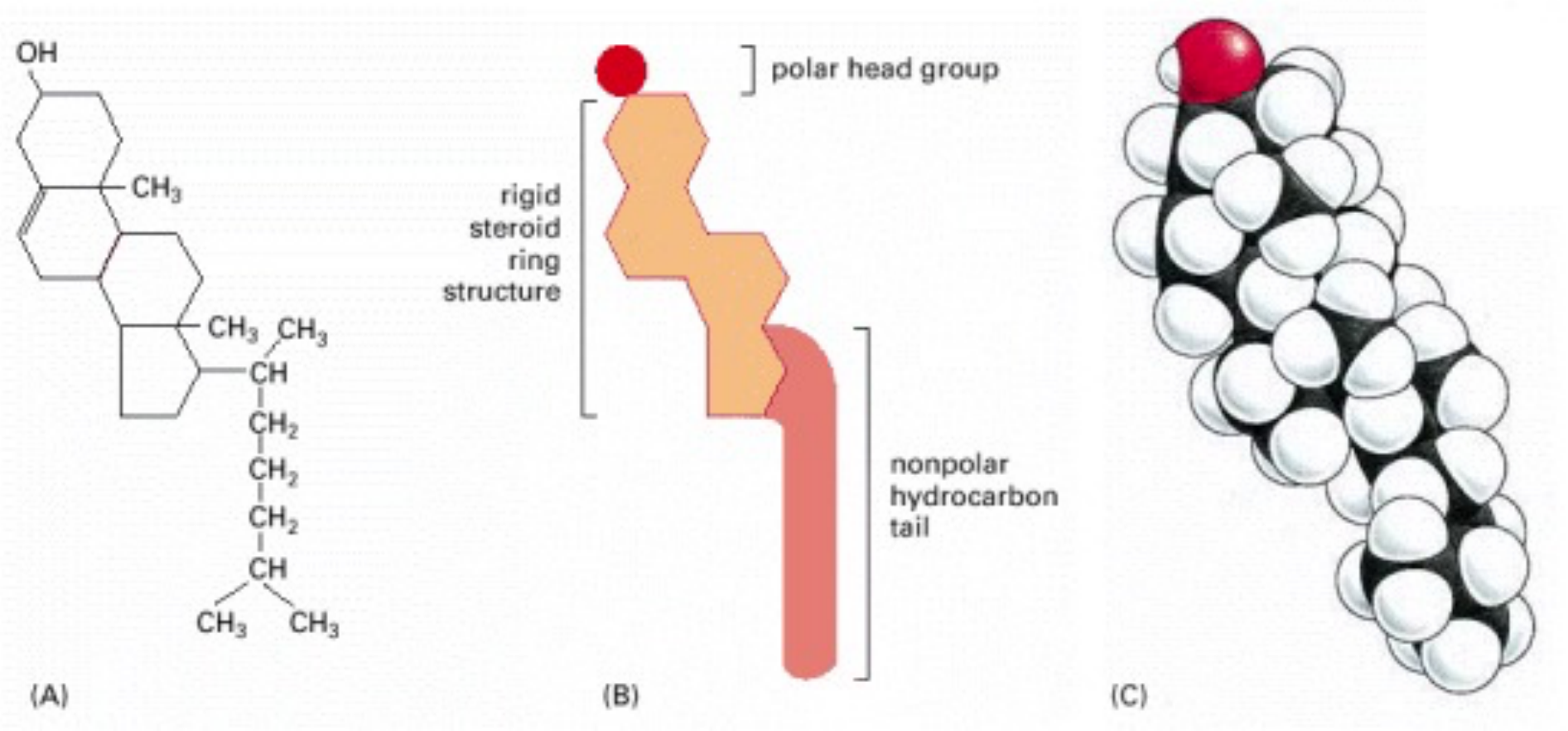}& \hspace{0.2in}
	\includegraphics[width=0.3\textwidth]{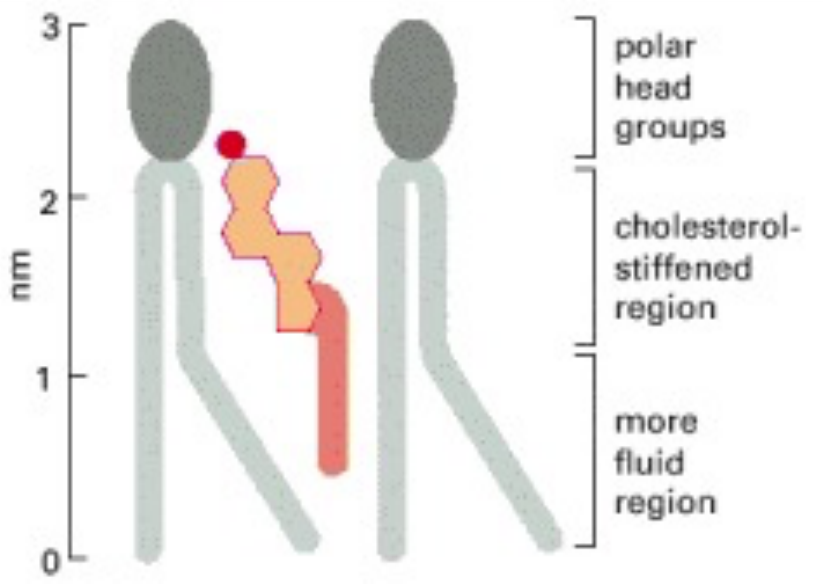}
	\end{tabular}
	\caption{(left) (a) Chemical composition, (b) schematic, and (c) volume rendering of cholesterol. (Right) Caricature of cholesterol residing with void space of a lipid bilayer. Its small head group serves to orient the molecule and its asymmetric shape provides leverage on the lipid's tail groups to constraint their range of motion. Reprinted with permission from \cite{MBoC-02}}\label{f:cholesterol-schematic}
\end{figure}

\kpc{In this paper} we introduce the multicomponent functionalized energy \kpc{(MCF}) as a general framework for a system of $n+1$ constituent species 
residing in a domain $\Omega\subset\R^d$, and provide a sharp characterization of the bilayer structures that are robustly stable to 
pearling bifurcations. The characterization involves only the spectrum of the linearization of the reduced dynamical system (\ref{e:FW}) 
that defines the connecting profile.  This framework contains the \kpc{subfamily} of two-component singularly 
perturbed systems that describe strongly asymmetric two-component mixtures. Previous work has exploited the
asymmetry to provide explicit leading order constructions of homoclinic connections \cite{DoelmanVeerman.2015}.  In Theorems\,\ref{t:H2coer} and 
\ref{t:FW_singpert_Evansf} we show that the robust  pearling stability condition corresponds to a natural geometric feature arising in the singular 
perturbation construction. We \kpc{present a minimal two component phospholipid-cholesterol bilayer (PCB) model arising from self-consistent mean field reductions of molecular models, \cite{CP-20p}, that captures} essential features of this 
ubiquitous system. In particular the PCB model encodes \kpc{the length imbalance between the cholesterol and phospholipid molecules, and the interdigitated packing that allows cholesterol to leverage an asymmetric influence on the phospholipid tails \cite{MS08}.  We conjecture that these asymmetries afford a mechanism  allowing cholesterol type molecules to robustly stabilize phospholipid membranes. }

\subsection{The Multicomponent Functionalized Energy}
The multicomponent functionalized (MCF) energy takes the form
	\begin{equation}\label{eq:FfCH}
	\cF[\bu] = \int_{\Omega} \frac{1}{2} \left|D^2\eps^2\Delta \bu -  \bF(\bu)\right|^2 - \eps^p \bP(\bu,\nabla \bu)\,\text{d} x,
	\end{equation}
with $\bu \in H^2(\Omega)$, $D$ is an $n\times n$, positive diagonal matrix, $\bF: \R^n\mapsto \R^n$ is a smooth vector field, and 
$\bP:\R^n\times \R^{d\times n} \mapsto \R$ represents the lower order functionalization term. This model generalizes the multicomponent
functionalized Cahn-Hilliard free energy introduced in \cite{PromislowWu.2017}, replacing the gradient form of the vector field with the more general function $\bF$ whose non-gradient form plays a central role in the generation of robust pearling stability.

The components $\{u_i\}_{i=1}^n$ of $\bu$ and $u_{n+1}:=1-u_1-\ldots -u_n$ represent the volume fractions of 
the $n+1$ constituent species. Each species is classified as either amphiphilic or solvent. There can be more than 
one solvent phase, in which case they are generally immiscible, \cite{MBarek17}. The zeros $\{\ba_i\}_{i=0}^m$  of 
$\bF$ are associated to pure solvent phases and act as rest-states for the system. 
The domains $\Omega_i:=\{x\in\Omega\,\big|\, |\bu(x)-\ba_i|=O(\eps)\}$ can have $O(1)$ volume without generating 
leading order contributions to the free energy.  The dominant term in the multicomponent functionalized energy encodes 
proximity to ``good packings'' of the molecules identified as solutions, or approximate solutions, of the packing relation: $D^2\eps^2\Delta \bu=\bF(\bu).$ 
The MCF energy is typically coupled with a non-negative linear operator $\cG$, called the gradient, that annihilates the constants. A canonical
choice is $\cG=-\Delta$.  The result is the gradient flow
\beq
\begin{aligned}
\bu_t &= -\cG \frac{\delta \cF}{\delta u},\\
\bu(0) & =\bu_0,
\end{aligned}
\eeq
where the variational derivative is taken with respect to the $L^2(\Omega)$ inner product. When combined with appropriate boundary
conditions, for example periodic boundary conditions, the result is a flow which decreases the energy $\cF[\bu(t)]$ while preserving the total mass of each constituent species. This work focuses on the properties of the energy, and constructions that lead to normally coercive low-energy manifolds
of $\cF$.

In section 2 we characterize the properties of connecting orbits that arise as the good packings that separate domains 
$\Omega_i$ and $\Omega_j$ with an $O(\eps)$ width interface comprised of amphiphilic molecules. 
We take the interface to be flat, and measure normal distance in the scaled variable $z(x):={\rm dist}(x,\partial \Omega_i)/\eps$, 
and drop the lower order functionalization term $\bP$, so that the connecting profiles can be characterized as minimizers of the codimension 
one reduced energy
\begin{equation}\label{eq:F1-def}
	\cF_1[\bu] := \int_\R \frac{1}{2} \left|D^2 \partial_z^2\bu -  \bF(\bu)\right|^2\text{d} z,
	\end{equation}
subject to the constraint $\bu-\phi_{ij}\in H^2(\R)$
where $\phi_{ij}:=\ba_j +(\ba_i-\ba_j)(1-\textrm{tanh}(z))/2$ satisfies $\phi_{ij}\to \ba_j$ as $z\to\infty$ and $\phi_{ij}\to \ba_i$ as $z\to -\infty$.
When they exist, the absolute minimizers are the orbits of the $2n$ dimensional dynamical system
\beq
\label{e:FW}
D^2 \partial_z^2 \bu - \bF(\bu)=0,
\eeq
that are heteroclinic (or homoclinic) to the zeros $\ba_i$ and $\ba_j$.   
These orbits are global minimizers of $\cF_1$, yielding zero energy. Correspondingly we call (\ref{e:FW}) the freeway system and the 
associated heteroclinic or homoclinic orbits the freeway connections.  When the diagonal elements of $D$ are strongly unequal, the freeway 
system fits within the framework of geometric singular perturbation (GSP) theory.   

The local minimizers of the reduced free energy, for which the quadratic residual is not zero, also provide relevant connections
between phases, especially when freeway connections do not exist. They satisfy
\beq
\label{e:FO}
\left(D^2\partial_z^2 -\nabla_u\bF(\bu)\right)^\dag\left(D^2\partial_z^2 \bu -\bF(\bu)\right)=0,
\eeq
subject the heteroclinic or homoclinic boundary conditions. Here $^\dag$ denotes $L^2$ adjoint, see subsection\,\ref{s:Notation}. 
This is a $4n$ dimensional dynamical system, and its solutions 
generically have non-zero reduced energy. We call this the toll-road system, and the associated heteroclinic orbits the toll-road connections. 
Even when matrix $D$ has the singularly perturbed structure, the toll-road system does not trivially fit within the 
classical singularly perturbed framework. However we show that toll-road connections are generically generated at saddle-node bifurcations of 
freeway connections, and characterize the energy of  the associated toll-road connection in terms of the saddle-node bifurcation parameter. 
These results establish the GSP theory as a powerful tool for the construction of MCF energies that support families of 
robustly stable connections with prescribed composition. 

In section 3 we extend the zero-energy, flat-interface, freeway connections generated by the GSP theory to low-energy, curved-interface
functions in $H^2(\Omega)$ through a dressing process, given in Definition\,\ref{d:dressing}. This allows the construction of a low-energy, 
freeway manifold  parameterized by underlying ``admissible interfaces'', given in Definition\,\ref{d:admissible}.  The main analytical result, 
Theorem\,\ref{t:H2coer}, characterizes homoclinic freeway connections for which the associated freeway manifold is normally coercive, independent
of $\eps$ sufficiently small. The principal loss of coercivity in scalar systems arises through the onset of the pearling bifurcation which
triggers a high-frequency modulation of the bilayer width that can lead to its break-up into structures with lower codimension, \cite{NP19}.
Indeed, the pearling bifurcation can be triggered dynamically by $O(\eps)$ changes in the bulk lipid density. 
Theorem\,\ref{t:H2coer} specifically rules out these classes of instability through a condition on the spectrum of the linearization, 
$\mrL$ of the homoclinic freeway connection about the freeway system (\ref{e:FW}), see (\ref{e:Ldef}), that is readily
verifiable within the GSP framework.  There is a significant literature that develops rigorous estimates on slow motion of 
gradient flow systems near low-energy manifolds, see  \cite{OR-07} and \cite{BFK-18}.  A key component of this analysis
is played by the uniform coercivity of the energy to perturbations normal to the manifold, that allow the derivation of the asymptotic evolution 
of the system in the tangent plane of the manifold.  In \cite{HP20} these slow flow results have been extended to recover leading order 
dynamics associated to the slow flow, and we believe that the results of Theorem\,\ref{t:H2coer} will allow the interfacial motion results in
 \cite{CP-20} to be extended rigorously to a wide class of gradient flows of the MCF energy near the low-energy freeway manifolds constructed 
 herein.

In Section 4 we examine the structure of the MCF energy in the neighborhood of a saddle node bifurcation of freeway 
homoclinics within the GSP framework. At the bifurcation point the kernel of $\mrL$ is not simple, rendering Theorem\,
\ref{t:H2coer} inapplicable. Modulo non-degeneracy assumptions Theorem\,\ref{t:tr_exist} shows that the freeway saddle node 
bifurcation induces a toll-road homoclinic and characterizes its energy as quadratic function of the distance of the bifurcation 
parameter past criticality. In particular, we give an explicit example of a freeway saddle node bifurcation within the PCB model,
characterizing the energy of the toll-road homoclinic in terms of the readily computable geometric features of the model.
          
The synergy between the MCF energy and the GSP theory is particularly fortuitous, as there is limited intuition 
for the relation between the structure of the nonlinearity in higher-order, multicomponent models and the physical properties of
the constituent molecules. 
Rigorous derivation of higher-order free energies from more fundamental models, such as the derivation of the 
Ohta-Kawasaki free energy from the self-consistent mean field theory, have been performed, see \cite{CR-03, CR-05} for a general
framework and \cite{Un09, UnDoi05} for models specific to surfactants. However the
analysis in such derivations is generically weakly nonlinear, and affords little information on nonlinear interactions beyond those 
imposed in an ad-hoc manner, generically through incompressibility arguments. 
The MCF energies constructed from the GSP approach are strongly nonlinear and strongly asymmetric in their nonlinear terms.  
This asymmetry plays an essential role in the analysis, rendering the operator $\mrL$ strongly non-selfadjoint and sweeping its spectrum off of the 
positive real axis and into  the complex plane. This complexification is stabilizing as neutral modes in the linearization of the MCH about
a homoclinic freeway connection arise from a balance between positive real spectrum of $\mrL$ against negative spectrum of 
the surface diffusion operator. 

The phospholipid-cholesterol bilayer model presented in Section\,\ref{s:PCB}, is the minimal GSP based model
that supports both a single-phase pearling-neutral phospholipid bilayer, and a two-phase phospholipid-cholesterol bilayer that is robustly stable
to pearling.  It is tempting to find synergy between the generic, geometric nature of these stability results and the 
generic presence of cholesterol within phospholipid membranes. Cholesterol's interdigitation between lipid molecules leads to a 
core density peak and an outsized impact on lipid mobility that inhibits the lipid tail compression required for the onset of pearling 
bifurcations, \cite{MS08}. It may be that the genome has latched upon the generic, geometric, singular role of cholesterol as a 
mechanism to prevent formation of micelles and other higher codimensional defects within phospholipid membranes.

\subsection{Notation}
\label{s:Notation}
Consider a function $f:\R\mapsto X$ where $X$ is a Banach space and $s\in\R$ is a parameter in $f$. We say that $f$ is $s$-exponentially small in $X$ 
if there exists  $\nu>0$ such that $\|f\|_X\leq e^{-\nu/s}$ for $s>0$ as $s\ll1$ tends to zero.

We use $^t$ to denote the transpose of a matrix or a vector in the usual Euclidean inner product and $^\dag$ to denote the an adjoint operator 
or eigenfunction in the $L^2(\Omega)$ inner product.

\section{Connecting Orbits}
In this section we establish the structure of the freeway and toll-road connection problems and the existence of specific solutions in
the context of the geometric singular perturbation scaling.
\subsection{Freeway and Toll-road connections}
\label{s:FTR_connections}
We assume that $\bF:\R^n\mapsto\R^n$ is smooth and has $m+1$ zeros $\ba_0,\cdots, \ba_{m}$ for which $\bF(\ba_i)=0,$
and $D$ is an $n\times n$, non-negative diagonal matrix. Generically the phase space is mapped onto species densities
with the variable $u_i$ denoting the volume fraction of species $i$, residing in
$$\cmD:=\left\{ \bu \, \Bigl|\, u_i\geq 0, i=1, \ldots, n, \sum_{i=1}^n u_i \leq 1\right\}.$$
Zeros  of $\bF$ denote the solvent phases, and when modeling a mixture with a single solvent it is generically taken as 
$\ba_0:=(0,\ldots, 0)$ with $\{u_1, \ldots, u_n\}$ denoting $n$ amphiphilic phases. In low energy configurations
these surfactants reside on thin interfaces generically of codimension one or  higher, that are $O(\eps)$ thin in 
one or more directions (the co-dimensions).  We focus on codimension one geometries, and in this section
fix the interface $\Gamma$ to be a flat $d-1$ dimensional hypersurface, so that the minimization problem reduces at leading order 
to the system  for $\cF_1$ given in (\ref{eq:F1-def}).
The infimum is non-negative and if attained, then the minimizer is smooth and satisfies the associated Euler-Lagrange equation (\ref{e:FO}), which we call the toll-road system.
Setting $\bG=D^{-2}\bF$, it is convenient to write the toll-road system as a $4n$ dimensional, first-order system,
  \beq
   \label{e:FO4n}
   \begin{aligned}
     \bu_z &=\bp,  \\
     \bp_z & = \bv+ \bG(\bu),\\
     \bv_z & = \bq, \\
     \bq_z &= \nabla_\bu \bG(\bu)^t \bv,
   \end{aligned}
   \eeq
   where as a consequence $\bv :=\bu_{zz}-\bG(\bu)^t.$  A zero $\ba$ of $F$ is normally hyperbolic if the linearization about the zero $\bA:=(\ba,0,0,0)^t\in\R^{4n}$ of (\ref{e:FO4n})  has 
   no purely imaginary eigenvalues.  

 \begin{lem}
 \label{l:Ham}
A zero $\ba\in\R^n$ of $\bF$ is normally hyperbolic (\ref{e:FO4n}) if and only if the $n\times n$ matrix $D^{-2}\nabla_\bu F(\ba)$ 
has no eigenvalues in the set $\R_-=(-\infty, 0]$. In this case $\bA=(\ba,0,0,0)^t$ has a $2n$ dimensional stable and $2n$ dimensional unstable manifold within (\ref{e:FO4n}).
 The system (\ref{e:FO4n}) has a conserved quantity $\bH:\R^{4n}\mapsto \R$ given by 
 \beq
 \label{eq:Ham-def}
    \bH(\bu,\bu_z,\bv,\bv_z):= \bu_z\cdot \bv_z - \frac12 |\bv|^2 - \bv\cdot D^{-2} \bF(\bu).
  \eeq
  In particular homoclinic and heteroclinic solutions of (\ref{e:FO4n}) lie on the $\R^{4n-1}$ dimensional $\{H=0\}$ level set. 
  Let $\ba_i$ and $\ba_j$ be  two normally  hyperbolic zeros of $\bF$ and  $\Phi=\Phi_{ij}(z;\gamma)$ be a smooth $k\geq1$ dimensional manifold of 
  connections between $\ba_i$ and $\ba_j$. Then  there exists $\alpha_{ij}\in\R_+$ such that $\cF_1(\Phi)=\alpha_{ij}$ for all connecting orbits $\Phi$ on the manifold.
 \end{lem}
 \begin{proof}
 Using the relation $\bv=\bu_{zz}-\bG(\bu)^t$
  we take the dot product with $\bu_z$
  \beq
  \bu_z^t\bv_{zz} - \bu_z^t \nabla_\bu \bG(\bu)^t\bv=0,
  \eeq
  and equivalently, since a scalar equals its own transpose, we have
  \beq
  \frac{d}{dz} \left( \bu_z^t \bv_z\right) - \bu_{zz}^t \bv_z - \bv^t \nabla_\bu \bG(\bu) \bu_z=0.
  \eeq
  Substituting $\bu_{zz}= \bv+\bG(\bu)$ we find
  \beq
  \frac{d}{dz}\bH(\bu,\bu_z,\bv,\bv_z)=0.
  \eeq
  where $\bH$ is as defined in (\ref{eq:Ham-def}). 
  
   Each of the zeros $\ba$ of $\bF$ satisfies $\bH(\ba,0,0,0)=0$, and since $\bH$ is conserved under the flow the  orbits connecting 
   these zeros reside on the $4n-1$ dimensional level set $\{\bH=0\}$.   If $\ba$ is a critical point of $\bG$ then the linearization of the system (\ref{e:FO4n}) about $\bA:=(\ba,0,0,0)^t$, takes the form $\bU_z = \mrM \bU$, 
   where
   \beq
   \mrM := \bpm 0 & I_{n\times n}&0 &0 \cr
                                                                             \nabla_\bu\bG(\ba) & 0 & I_{n\times n}& 0\cr
                                                                             0&0&0& I_{n\times n} \cr
                                                                             0&0& \nabla_\bu\bG(\ba)^t& 0 \epm.
    \eeq          
 We compute that $\lambda\in\sigma(\nabla_\bu (\bG(\ba)))$ if and only if the four values $\{ \pm\sqrt{\lambda}, \pm\sqrt{\lambda^*}\}$ lie in $\sigma(\mrM)$ up to algebraic multiplicity.  In particular the isolated critical point $\ba$ of $\bF$ is normally hyperbolic within the toll-road system if and only if $D^{-2}\nabla_\bu (\bF(\ba))$ has no purely real, negative eigenvalues. If $\ba$ is normally hyperbolic, the symmetry of the spectrum of $\mrM$ guarantees that the stable and unstable manifolds of the $4n$ dimensional system has equal dimension, hence they are both $2n$ dimensional.

 To establish the uniformity of the energy over the manifold of connections, we insert $\Phi_{ij}$ into (\ref{e:FO}) are 
 rewrite it in the form
  \beq
  \label{e:FOe}
  D^2\partial_z^2\Phi_{ij} -\bF(\Phi_{ij}) =\tPsi_{ij}(\gamma),
  \eeq
  where the right-hand side, $\tPsi_{ij}$, lies in the kernel of the adjoint of 
  \beq
  \mrL := D^2\partial_z^2 -\nabla_\bu \bF(\Phi_{ij}).
  \eeq
  Taking the partial of (\ref{e:FOe}) with respect to $\gamma$ yields the relation
  \beq
  \mrL \frac{\partial \Phi_{ij}}{\partial \gamma} = \frac{\partial \tPsi_{ij}}{\partial \gamma}.
  \eeq
  This implies that the right-hand side is $L^2(\R)$ orthogonal to the kernel of $\mrL^t$ and hence to $\tPsi_{ij}.$ In particular $\partial_\gamma \left\|\,\tPsi_{ij}\,\right\|_{L^2}^2=0$
  and we may write $\tPsi_{ij}(\gamma)=\alpha_{ij}\Psi_{ij}(\gamma)$ where $\|\Psi_{ij}(\gamma)\|_{L^2}=1$. The result follows since 
  $\cF_1(\Phi_{ij})= \|\tPsi_{ij}\|_{L^2(\R)}= \alpha_{ij}$ is independent of $\gamma$.
  \end{proof}

We reinforce this dichotomy of zero-energy and non-zero energy connections through the following definition.
\begin{defn}
If a manifold of connections $\Phi_{ij}$ has zero energy, $\alpha_{ij}=0$, then the constituent orbits satisfy the freeway sub-system
(\ref{e:FW}). 
We call these orbits freeway connections.
 \end{defn} 

\subsection{Freeway homoclinic connections in singularly perturbed systems}\label{s:FW_singpert_existence}
\label{s:FHC_SPS}
Establishing the existence of connections in $n$-dimensional dynamical systems of the general form \eqref{e:FW} is nontrivial. 
However, when the eigenvalues of the matrix $D$ exhibit a wide range of magnitudes, controlled by a
small parameter $0 < \delta \ll 1$, then the associated dynamical system may have orbits that can be rigorously 
constructed via geometric singular perturbation theory by gluing together solutions of the so-called slow and fast 
sub-systems of reduced dimension. In \cite{DoelmanVeerman.2015}, theory was developed that provides for the existence and 
spectral analysis of homoclinic connections in a general class of two-component, singularly perturbed vector fields for the case in which the vector field is strongly non-symmetric. The homoclinic connection problem is equivalent to the freeway system \eqref{e:FW} 
with $n=2$, and $D = \text{diag}(1,\delta)$, where $0<\delta \ll 1$, and the vector field $\bF$ takes the form
\begin{equation}\label{e:FW_singpert_Fdelta}
\bF(\bu;\delta) = \bpm
\bF_{11}(u_1;\delta) + \frac{1}{\delta} \bF_{12}(u_1,u_2;\delta) \\ \bF_2(u_1,u_2;\delta)
\epm.
\end{equation}
The component functions $\bF_{ij}$ obey mild regularity assumptions \cite{DoelmanVeerman.2015}. 
The resulting model can be written as a first order dynamical system in the form (\ref{e:FO}),
which in the fast spatial variable $\zeta := z / \delta$  takes the form
\begin{equation}\label{e:FW_singpert_system_zeta}
\begin{aligned}
(u_1)_\zeta & = \delta\,p_1,  \\
(p_1)_\zeta & = \delta\,\bF_{11}(u_1;\delta) + \bF_{12}(u_1,u_2;\delta),\\
(u_2)_\zeta & = p_2, \\
(p_2)_\zeta &= \bF_2(u_1,u_2;\delta).
\end{aligned}
\end{equation}
 We require that the following structural assumptions hold:
\begin{ass}
	\label{a:A2}
	The point $\ba = (0,0)$ is an isolated, hyperbolic zero of \eqref{e:FW}. The component functions satisfy $\bF_{12}(u_1,0; \delta) = 0$ and $\bF_2(u_1,0; \delta)=0$ for every $u_1 \in \R$. 
	There exists an open set $V \subset \R$, such that the planar system $(u_2)_{\zeta\zeta} - \bF_2(s,u_2;0) = 0$ admits a symmetric solution $u_{2,h}(\zeta;s)$ that is homoclinic to $u_2=0$ for every $s \in V$. 
\end{ass}
\begin{rmk}
The parameter $\delta$ is taken asymptotically small in the GSP theory, however in the context of the MCF 
energy it denotes the ratio of lengths comparable molecules, and is not vanishingly small. 
Correspondingly we take $\delta$ sufficiently small to apply the GSP theory, but then
consider it to be a fixed parameter in the subsequent analysis of the MCF energy. 
In particular the upper bound $\eps_0$, on the value of admissible $\eps$ will depend upon the 
fixed value of  $\delta$. In effect the GSP theory applies in the regime $\eps\ll \delta\ll 1,$ as is consistent with applications. 
\end{rmk}
From these assumptions, it directly follows that the origin of  \eqref{e:FW_singpert_system_zeta} is a hyperbolic equilibrium. 
Moreover, we see that the manifold $\cmM_0 := \left\{u_2=p_2=0\right\}$ is invariant under the flow of \eqref{e:FW_singpert_system_zeta}. 
The flow on $\cmM_0$, which we call the reduced slow flow, is given to leading order in $\delta$ in the slow variables by 
\begin{equation}\label{e:FW_singpert_slowsub}
(u_1)_{zz} = \bF_{11}(u_1;0).
\end{equation}
From the assumptions, the point $(u_1,p_1) = (0,0)$ is a hyperbolic equilibrium of \eqref{e:FW_singpert_slowsub} and the associated (slow) stable and 
unstable manifolds $\Wuss \subset \cM_0$ are one-dimensional, and equal to the other's reflection about the $u_1$-axis; see Figure \ref{fig:FW_singpert_ex}. 
Defining $u_{1,s}^\mathrm{s}(z;s)$ as the unique positive solution to \eqref{e:FW_singpert_slowsub} satisfying $u_{1,s}^\mathrm{s}(0;s) = s$ and $\lim_{z\to\infty} u_{1,s}^\mathrm{s}(z;s) = 0$, we see that the one-dimensional (slow) stable manifold of the origin $\Wss \subset \cM_0$ for $u_1 > 0$ is given by the orbit of $u_{1,s}^\mathrm{s}$. Moreover, both the stable and unstable manifolds, $\Wuss$, lie on the level set 
$\left\{(u_1,p_1) : \frac{1}{2}p_1^2 + \int_0^{u_1} \!\bF_{11}(\hat{u}_1;0)\,\text{d}\hat{u}_1 = 0\right\}$.
	
Conversely, in the fast scaling \eqref{e:FW_singpert_system_zeta}, we see that to leading order in $\delta$, $u_1=s$ is constant, 
$$p_1(\zeta) = p_0 + \int_0^{\zeta} \!\bF_{12}(s,u_2(\hat{\zeta});0)\,\text{d}\hat{\zeta},$$ 
while $u_2$ obeys the so-called reduced fast flow  
\begin{equation}\label{e:FW_singpert_fastsub}
(u_2)_{\zeta\zeta} - \bF_2(s,u_2;0) = 0. 
\end{equation}
The manifold $\cmM_0$ is exactly the set of trivial equilibria of \eqref{e:FW_singpert_fastsub}; by the requirements of Assumption \ref{a:A2}, these trivial equilibria are hyperbolic.  Moreover, there exists an open subset $\mathcal{M}_1 \subset \cM_0$, $\cM_1 := \left\{u_2=p_2=0,\,u_1\in V\right\}$, such that the reduced fast flow connects $(s,p_{1,o},0,0) \in \cM_1$ to $(s,p_{1,d},0,0) \in \cM_1$ through the symmetric fast homoclinic orbit $u_{2,h}(\zeta;s)$. In this reduced limit, the jump in $u_1$'s derivative satisfies
$$\Delta p(s):= p_{1,d} - p_{1,o} = \int_\R \!\bF_{12}(s,u_{2,h}(\zeta;s);0)\,\text{d}\zeta.$$ 
This suggests the definition of a pair of curves on $\cM_0$, called the ``take-off'' and ``touchdown'' curves,
for which $\Delta p$ transports the orbit first away from, and then back to $\cmM_0$ in a symmetric fashion.
The take-off curve is given by $\To := \left\{p_1 = -\frac{1}{2}\Delta p(u_1) \right\}$, while the touchdown curve $T_\mathrm{d}$ is given by its reflection about the $u_1$-axis; see Figure \ref{fig:FW_singpert_ex}.

\begin{figure}
	\centering
	\includegraphics[width=0.5\textwidth]{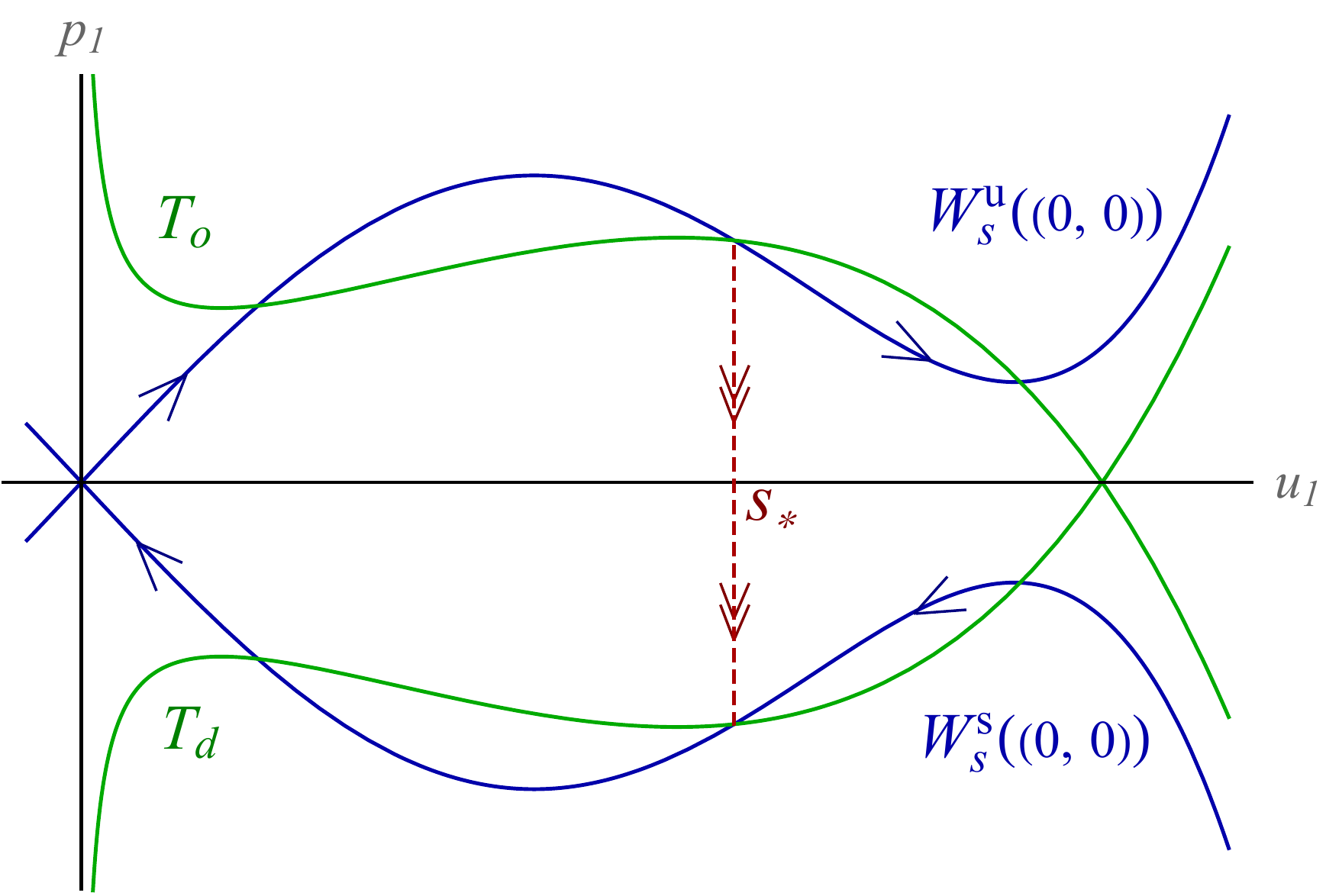}
	\caption{A schematic representation of the reduced slow flow on $\cM_0$. The slow stable and unstable manifolds $\Wuss$ are indicated in blue, the take-off and touchdown curves $\To$ and $T_\mathrm{d}$ are indicated in green. The jump through the fast field at a transversal intersection of $\Wus$ and $\To$ for 
$u_1 = s_*$ is indicated in red.} \label{fig:FW_singpert_ex}
\end{figure}

A homoclinic orbit of the GSP scaling of the $4$-dimensional freeway problem lies in the transversal (first) intersection of the 2-dimensional stable 
and unstable manifolds  of the origin $(0,0,0,0)$. The scale separation present in the system allows us to decompose this intersection into a first 
slow segment  that follows $\Wus\subset \cM_0$ closely, then makes a fast excursion away from $\cM_0$, but $\mathcal{O}(\delta)$ close to $u_{2,h}(\zeta,s_*)$ for some 
$s_*$, and then touches down again near $\cM_0$ to closely follow $\Wss\subset \cM_0$ back to the origin $(0,0,0,0)$. In the singular limit, this 
concatenation procedure provides a homoclinic orbit precisely when the take-off curve $\To$ intersects the slow unstable manifold $\Wus$; see Figure \ref{fig:FW_singpert_ex}. When this intersection is nongenerate, transversality arguments imply that the singular orbit persists for sufficiently small 
$0<\delta \ll 1$; for the full analysis, see \cite{DoelmanVeerman.2015}.

We define the function $\rho:V \mapsto \R$
\begin{equation}\label{e:FW_singpert_rho}
 \rho(s) := \int_0^{s} \!\!\!\!\bF_{11}(\hat{u}_1;0)\,\text{d}\hat{u}_1 - \frac{1}{8} \left(\int_\R \!\bF_{12}(s,u_{2,h}(\zeta;s);0)\,\text{d}\zeta\right)^2.
\end{equation}
One can deduce that if the take-off curve $\To$ and the slow unstable manifold $\Wus$ intersect transversally at $u_1 = s_*$, then the 
function $\rho$ has a nondegenerate root at $s=s_*$. However, $\rho$ also vanishes when $\To$ intersects the slow \emph{stable} manifold $\Wss$, which does 
not lead to a meaningful geometric construction when $\Wus \cap \Wss = \emptyset$. To exclude these spurious roots, we employ the explicit characterisation of $\Wus$ by the solution $u_{1,s}^\mathrm{s}$, and introduce the condition that $\text{sgn } \frac{\text{d} u_{1,s}^\mathrm{s}}{\text{d} z}(0;s_*) = \text{sgn }\frac{1}{2} \Delta p(s_*)$.\\

The following is a reformulation of \cite[Theorem 2.1]{DoelmanVeerman.2015}.

\begin{thm}
\label{t:GSP_SS} Assume $n=2$, $D = \mathrm{diag}(1,\delta)$, that $\bF$ takes the form \eqref{e:FW_singpert_Fdelta}, and that the conditions of Assumption\,\ref{a:A2} hold. Fix $\delta>0$ sufficiently small. Let $N$ denote the number of nondegenerate roots of $\rho$, defined in \eqref{e:FW_singpert_rho}, that lie in the set $V$, and that obey the condition
\begin{equation}\label{e:FW_thm_condition}
\frac{\text{d} u_{1,s}^\mathrm{s}}{\text{d} z}(0;s_*) \int_\R \!\bF_{12}(s_*,u_{2,h}(\zeta;s_*);0)\,\text{d}\zeta > 0.
\end{equation}
Here, $u_{1,s}^\mathrm{s}(z;s)$ is the unique positive solution to \eqref{e:FW_singpert_slowsub} satisfying $u_{1,s}^\mathrm{s}(0;s) = s$ and $u_{1,s}^\mathrm{s}(z;s) \to 0$ as $z\to\infty$. Then there are $N$ symmetric, positive, one-circuit solutions to \eqref{e:FW} that are homoclinic to $\ba=(0,0)$.
In particular, for each root $s_*$ the associated homoclinic connection
$(u_{1,*}(z),u_{2,*}(z))$, translated to be even about $z=0$, has the following spatial structure:
\begin{enumerate}
	\item for $0 \leq z < \!\!\sqrt{\delta}$, $u_{1,*}(z) = s_*$ and $u_{2,*}(z) = u_{2,h}(z/\delta;s_*)$ to leading order in $\delta$;
	\item for $\!\!\sqrt{\delta} \leq z$, $u_{1,*}(z) = u_{1,s}^\mathrm{s}(z;s_*)$ to leading order in $\delta$, while $u_{2,*}(z)$ is $\delta$-exponentially small.
\end{enumerate}
\end{thm}

\begin{rmk} The result from \cite{DoelmanVeerman.2015} encompasses a larger class of systems, in particular the 
zero $\ba$ may lie on the boundary of the domain of definition of the vector field $\bF$. This necessitates 
additional technical assumptions on $\bF$, see \cite[Assumptions (A1-4)]{DoelmanVeerman.2015}.
\end{rmk}

\subsection{A minimal phospholipid-cholesterol model}\label{s:PCB_model}
\label{s:PCB}
\kpc{The singularly perturbed framework presented in section \ref{s:FW_singpert_existence} encompasses a model of a 
phospholipid-cholesterol bilayer (PCB) membrane. This is a minimal form of the self-consistent mean field reductions of a blend of phospholipid, cholesterol, and solvent, \cite{CP-20p}. In the absence of cholesterol the PCB energy supports a pearling neutral bilayer membrane which becomes robustly stable when optimally loaded with cholesterol.  The PCB model takes the form}
\begin{equation}\label{e:PCB_model}
\bF_\mathrm{PCB}(\bu;\delta) = \bpm W'(u_1) - \frac{1}{3\delta} f(u_1)^2 u_2^2 \To(u_1)\\ u_2-f(u_1) u_2^2
\epm,
\end{equation}
where $f$ is a smooth, positive, non-increasing function. The slow component, $u_1$, denotes the volume fraction of phospholipid 
while the fast component, $u_2$, denotes that of cholesterol. The take-off curve $\To$ --\kpc{which is derived within the self-consistent mean field reduction}-- is smooth and specified below. The scalar potential $W$ is precisely the smooth double-well from \eqref{e:scalar-FCH} with minima at $b_-=0$ and $b_+=1$ satisfying $W(0)=0>W(1)$. In particular $W'(u_1)$ has a unique transverse zero $u_1=u_{1,\rm max}$ that lies in $(0,1)$, so that the slow stable and unstable manifolds $\mathcal{W}_s^{\mathrm{u,s}}((0,0))$ coincide for $u_1>0$, leading to the existence of a fully slow homoclinic orbit on $\cM_0$. Moreover, since $f$ is positive everywhere, the fast subsystem \eqref{e:FW_singpert_fastsub} admits a homoclinic orbit for every value of $u_1=s$, hence we may take $V=\R$ in Assumption \ref{a:A2}.
  
The homoclinic orbit $u_{2,h}$ in the fast system \eqref{e:FW_singpert_fastsub} has the explicit expression
\begin{equation}\label{e:PCB_u2h}
u_{2,h}(\zeta;s) = \frac{3}{2 f(s)} \, \text{sech}^2 (\zeta/2),
\end{equation}
allowing us to evaluate the integral
\begin{equation}\label{e:Cholest_F12int}
\Delta p(s) = \int_\R \bF_{12}(u_1,u_{2,h}(\zeta;s); 0)\,\text{d}\zeta = -\frac{f^2(s)\To(s)}{3}\int_\R u_{2,h}^2\,\text{d}\zeta = -2\To(s).
\end{equation}
Since $\bF_{11}(u_1; 0)=W'(u_1)$ and $W(0)=0$ the first integral in (\ref{e:FW_singpert_rho}) takes the values $W(s)$. Moreover, the portion of the unstable slow manifold in the first quadrant can be given as the graph $\left\{ (s,\omega_u(s))\, \bigl|\, s\in(0,u_1)\right\}$, where $\omega_u(s):=\sqrt{2W(s)}$.  
We calculate that
\begin{equation}\label{e:rho_PCB}
\rho_\text{PCB}(s) = W(s) - \frac{1}{2} \To(s)^2 =\frac12\left( \omega_u(s)- \To(s)\right)\left( \omega_u(s)+ \To(s)\right).
\end{equation}
As established in section 2.2, the  zeros of $\rho_\text{PCB}$ correspond to the crossings of the take-off curve with the graph of the 
unstable slow manifold. We choose the take-off curve
to have a transverse intersection with the unstable manifold at the phospholipid density $u_1=s_*$ corresponding to a bilayer membrane fully interdigitated with
cholesterol. The cholesterol density is modulated by adjusting the value of $f(s_*),$ see (\ref{e:PCB_u2h}). 
For the slow subsystem, the slow stable and unstable manifolds of the origin 
$\Wuss$ coincide, so that condition \eqref{e:FW_thm_condition} is automatically satisfied -- 
that is, \emph{every} root of $\rho(s)$ is a valid candidate for the construction outlined in section \ref{s:FW_singpert_existence}. In Figure \ref{fig:cholesterol}, the dynamics on $\cM_0$ and a corresponding pulse are shown for a specific choice of $\To$. 
 
 \begin{figure}
 	\centering
	\includegraphics[width=0.45\textwidth]{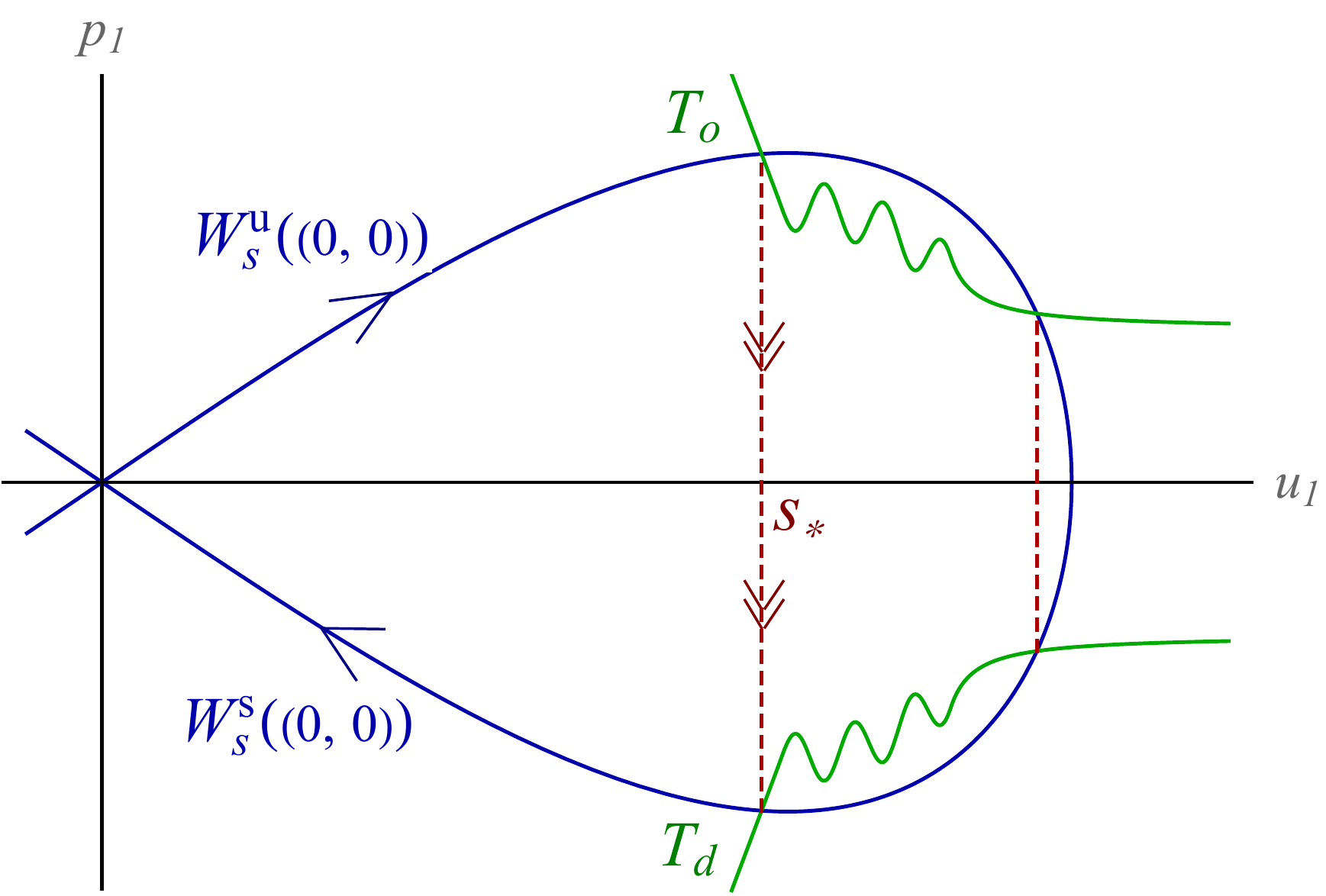}
	\hspace{0.1in}
	\includegraphics[width=0.4\textwidth]{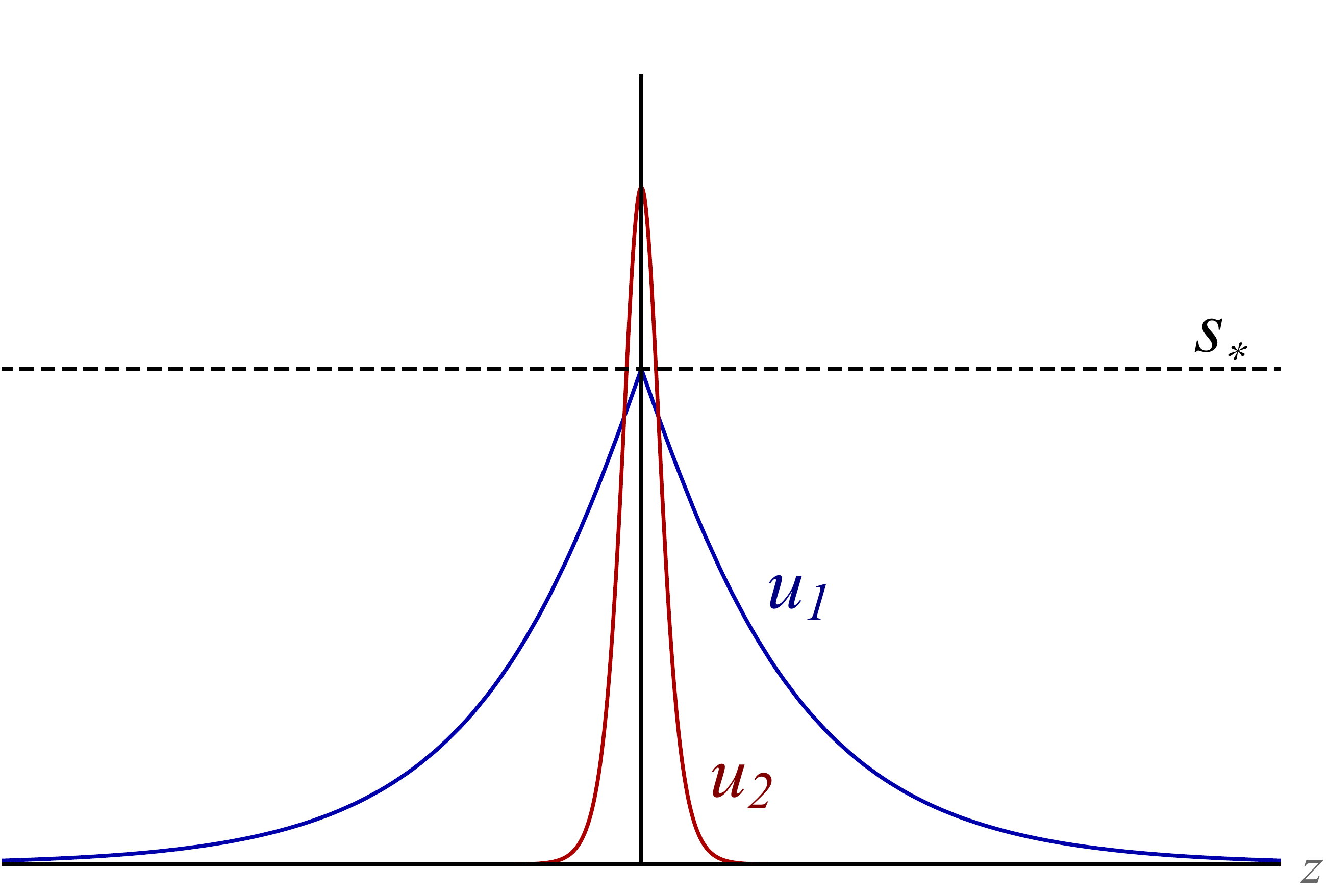}
	\caption{(left) The phospholipid cholesterol bilayer system \eqref{e:PCB_model} has two fast-slow homoclinic freeway connections, corresponding to the two intersections of the take-off curve $\To$ with the unstable manifold $\Wus$ of the origin of the slow system $(u_1)_{zz} = W'(u_1)$ \eqref{e:FW_singpert_slowsub}.
	The leftmost intersection, at $u_1=s_*$ has $\rho'(s_*)>0$ and will yield robustly stable bilayer interfaces. The rightmost intersection yields a fast slow connection with $\rho'<0$. The system also supports a slow-only homoclinic with $u_2\equiv 0.$ (right) Depiction of the fast-slow homoclinic connection for the $u_1=s_*$ intersection, corresponding to a cross section of a bilayer membrane with phospholipid ($u_1$) on the outside and interdigitated cholesterol ($u_2$) in the core. The maximum value of the slow component occurs at the take-off intersection point. The fast component $u_2=u_{2,h}(\zeta)$ is scaled by $f(s_*)$.
	}\label{fig:cholesterol}
\end{figure}

\section{Normal Coercivity of  Homoclinic Freeway Manifolds}
In this section we extend the freeway connections, generating the freeway manifold of low energy solutions associated
to a wide class of admissible interfaces. We identify conditions which guarantee the normal coercivity of homoclinic freeway manifolds,
and relate the stability conditions back to the construction of the homoclinic freeway connections within the GSP context.
\subsection{Freeway Manifolds}
\label{s:FM}
We consider a codimension one interface  $\Gamma$ given by the local parameterization
$$ x=\gamma(s)+\eps z\nu(s),$$
with $\gamma:\bS\subset \R^{d-1} \mapsto \Gamma\subset\Omega$ and $\nu(s)$ the outward normal to $\Gamma$ at $\gamma(s)$.
The pair $(s,z)$ forms a local coordinate system, for which $s=s(x)$ parameterizes location on $\Gamma$ and $z=z(x)$ is the $\eps$-scaled signed distance
to $\Gamma.$ The line segments $\{ \gamma(s)+ t\nu(s) \,\bigl|\, |t|<\ell\}_{s\in\bS}$ are called the whiskers of length $\ell$ of $\Gamma$. 
\begin{defn}
\label{d:admissible}
For any $K,\ell>0$ the family $\cG_{K,\ell}$ of admissible interfaces is comprised of closed (compact and without boundary),
oriented $d-1$ dimensional manifolds $\Gamma$ embedded in $\Omega\subset \R^d$, which are far from self-intersection and posses a smooth second
fundamental form. More precisely, $2\ell K<1$, the $W^{4,\infty}(\bS)$ norm of the principal curvatures of $\Gamma$ is bounded by $K$, the whiskers of length $2\ell$ do not intersect each-other nor the boundary of $\Omega$, and the surface area, $|\Gamma|$, 
of $\Gamma$ is bounded by $K$.  We call the set $\Gamma_{2\ell}:=\{x\, \bigl| \, |z(x)|<2\ell/\eps\},$  the reach of $\Gamma.$
\end{defn} 
For an admissible $\Gamma$ the change of variables $x\mapsto (s,z)$ is a $C^4$ diffeomorphism of $\Gamma_{2\ell}$, 
see section 6 of \cite{HayrapetyanPromislow.2015} . To each class $\cG_{K,\ell}$ we associate a symmetric, compactly supported $\mcC^\infty$  function $\xi$ that is monotone on $\R_+$ and takes values $1$ on $[-\ell,\ell]$ and is 0 on $\R\backslash [-2\ell, 2\ell]$.
\begin{defn}
\label{d:dressing}
For each function $\bu$ which tends at an exponential rate to constant values $\bu\to\bu_\pm$ as $z\to\pm\infty$,   we define its $\Gamma$-dressing as the $L^2(\Omega)$ function
$$\bu_\Gamma(x):= \bu(z) \xi(z\eps) + \oxi(z\eps) \bu_{\rm{sign}(z)},$$
where $\oxi:=1-\xi.$
\end{defn}

We will denote both the $\Gamma$-dressing and the original $L^2(\R)$ function by $\bu$ where doing so does not introduce confusion. 
The function $\xi(\eps z(x))$ lies in $H^4(\Omega)$, even though the distance function $z$ is not smooth outside the set $\Gamma_{2\ell}.$ 

\begin{defn}
\label{d:freeway-manifold}
To each freeway connection $\bu_*$ of the subsystem (\ref{e:FW}) and admissible family of interfaces
$\cG_{K,\ell}$ we associate the corresponding freeway manifold 
\beq
\label{e:FW_manifold}
\cM_{K,\ell}(\bu_*) := \{ \bu_{*,\Gamma}\, \bigl | \, \Gamma\in \cG_{K,\ell}\}.
\eeq
comprised of the dressings $\bu_{*,\Gamma}\in L^2(\Omega)$ of the admissible interfaces by $\bu_*.$
\kpc{We will drop the $*$ subscript on $\bu_{*,\Gamma}$ where doing so does not generate ambiguity.}
\end{defn}

On the reach of $\Gamma$ the $(s,z)$ coordinate system induces a representation of the Cartesian Laplacian (denoted $\Delta_x$ to avoid ambiguity) in the form
\beq
\label{e:inner-Lap}
\eps^2\Delta_x = \partial_z^2 + \eps \kappa(s,z)\partial_z + \eps^2 \Delta_s +\eps^3 z D_{s,2},
\eeq
where $\kappa(s,z)=H(s)+O(\eps z)$ is an extension of the mean curvature $H(s)$ of $\Gamma$, $\Delta_s$ is the Laplace-Beltrami operator associated to $\Gamma$, and $D_{s,2}$ is a second order operator in $\nabla_s$ with coefficients whose $W^{4,\infty}$ norm is bounded by $K$, 
see Proposition 6.6 of \cite{HayrapetyanPromislow.2015}.  The eigenfunctions of $\Delta_s$ are given by $\{\theta_j\}_{j=0}^\infty$ with  eigenvalues $\{-\beta_j^2\}_{j=0}^\infty$ which satisfy $0=\beta_0 <\beta_1\leq \beta_2 \cdots,$ see \cite{Lich-58}. 

For functions supported in $\Gamma_{2\ell}$, the $L^2$ inner product takes the inner form
\beq
\label{e:L2-inner}
\langle \bf, \bg\rangle_2  := \int_{-2\ell/\eps}^{2\ell/\eps} \int_\bS \bf(z,s)\cdot\bg(s,z)J(s,z)\,ds\,dz,
\eeq
where $J$ is the Jacobian of the change of coordinate map from $x$ to $(z,s)$. In particular $J(s,z)=J_0(s)\tJ(s,z)$ where $J_0$ is the square root of the determinant of the first fundamental form 
of $\Gamma$  and $\tJ$ admits the expansion
\beq
\label{e:tJ-def}
 \tJ(s,z) = \eps \prod\limits_{i=1}^{d-1} (1-\eps z k_i)=\eps \sum\limits_{j=0}^d (\eps z)^{j}K_j(s),
 \eeq
where $k_1, \ldots, k_{d-1}$ are the principal curvatures of $\Gamma$, while $K_0=1$ and for $i=1,\ldots d-1$, $K_i=K_i(s)$ are $(-1)^i$ times the sum of the all products of $i$ curvatures of $\Gamma$. From the condition
$2\ell K<1$ for interfacial admissibility we deduce that 
\beq
\label{e:tJ-equiv}
\eps 2^{-(d-1)} \leq \tJ\leq \eps \frac{3}{2}^{d-1}, 
\eeq
and hence after a scaling by $\eps$,  the inner product
\beq
\langle \bf, \bg\rangle_{2,J_0}  := \int_{-2\ell/\eps}^{2\ell/\eps} \int_\bS \bf(z,s)\cdot\bg(s,z)J_0(s)\,ds\,dz,
\eeq
induces a norm equivalent to the usual $L^2$ norm on $\Gamma_{2\ell}.$ The Laplace-Beltrami eigenmodes are orthonormal in the inner product
\beq
\label{e:sinner}
\langle \alpha, \beta \rangle_s :=  \int_\bS \alpha(s)\beta(s)J_0(s)\,ds.
\eeq

\begin{prop}
\label{p:FW_manifold}
Let $\bu_*$ be a freeway connection between two zeros $\ba_-$ and $\ba_+$ of $\bF$.
The freeway manifold associated to $\bu_*$ lies in $H^2(\Omega)$. If the functionalization term within
the multicomponent functionalized energy (\ref{eq:FfCH}) satisfies 
$\bP(\ba_\pm,0)=0$, then  
\beq
\obP(s;\Gamma):= \int_\R P(\bu_*,\nu(s)\partial_z\bu_*)\,dz,
\eeq
is finite, and the dressings $\bu=\bu_\Gamma\in\cM_{K,\ell}(\bu_*)$ have leading order energy
\beq
\cF[\bu] = \eps^3  \|D\partial_z \bu_*\|_{L^2(\R)}^2 \int_\bS |H(s)|^2\, J_0(s)\, ds + \eps^{p+1} \int_{\bS} \obP(s;\Gamma)\, J_0(s)ds +O(\eps^4,\eps^{p+2}),
\eeq
where $H$ denotes the mean curvature and $\nu$ the outer normal of $\Gamma$.
\end{prop}
The proof of this result is a direct modification of prior results, see \cite[eqn (3.3) and Proposition 4.1]{HP11}, and is omitted.
\subsection{Normal coercivity}
\label{s:NC}
In the sequel we assume that $\ba_0=0$ is a hyperbolic zero of the vector field $\bF$, and 
that $\bu_*$ is a freeway  connection homoclinic to $\ba_0$. Without loss of generality we may set the
functionalization term $\bP$ equal to zero as it lower order in $\eps$ and does not impact the
$\eps$-uniform coercivity bounds we seek in (\ref{e:H2coer}), see also Remark\,\ref{r:Pzero}.
We let $\mrL$ denote the linearization of the freeway system (\ref{e:FW}) about $\bu_{*}$. 
\beq
\label{e:Ldef} 
\mrL:= D^2 \partial_z^2 - \nabla _\bu \bF(\bu_*).
\eeq
 \kpc{We fix $K, \ell>0$, choose $\Gamma\in\cG_{K,\ell}$, and let $\bu_{*,\Gamma}$ denote the dressing of $\Gamma$ by the 
 homoclinic freeway connection $\bu_*$. For notational simplicity we drop the $*$ subscript on 
 $\bu_{*,\Gamma}$ in the remainder of this section.}

Our goal is to derive estimates on lower bounds of the spectrum of the second variational derivative of $\cF$ at $\busg$, 
given by the operator
\kpc{
\beq\label{e:mbL-def}
\mbL=\mbL(\busg):=\frac{\delta^2 \cF}{\delta \bu^2}(\busg)=\mbL_1 +\mathrm R, \qquad \mbL_1:= \cL^\dag\cL
 \eeq
 where we have introduced
 \beq
 \label{e:cL-def} 
  \cL:=\eps^2 D^2\Delta_x - \nabla_\bu \bF(\busg), \qquad \mathrm R:=- \nabla_\bu^2 \bF(\busg)\cdot  \left(\eps^2 D^2\Delta \busg-\bF(\busg)\right).
 \eeq
 We establish the uniform coercivity of the factored operator $\mbL_1$, the remainder term $\mathrm R$ is an asymptotically weak multiplier operator that is significant when addressing neutrally stable bilayers, but has minor impact on our analysis here. Indeed, outside the reach of $\Gamma$, $\busg$ is a spatially constant zero of $\bF$, and hence $\mathrm R$ is identically zero there.  Within the reach, combining the inner expansion \eqref{e:inner-Lap} of the Laplacian with the fact that  $\bu_*$ satisfies the freeway system \eqref{e:FW}, we see that $\mathrm R$ reduces to
 \beq
 \mathrm R = -\eps\kappa(s,z) \nabla_\bu^2\bF(\busg)\cdot D^2\partial_z \busg + O(\eps^2),
 \eeq
where $\kappa$ is the curvature of $\Gamma$. We deduce that
\beq\label{e:R}
\|\mathrm R\|_{L^\infty(\Omega)}\leq C  \eps
\eeq
for some positive constant $C$ depending only upon $K,\ell$ and the smoothness of $\bu_*$.}
In addition  within the reach of $\Gamma$ the operator $\cL$ admits the exact expansion
\beq
\label{e:cL-reach}
 \cL = D^2\left(\partial_z^2 +\eps \kappa(s,z) \partial_z+ \eps^2\Delta_s +\eps^3 z D_{s,2}\right) - \nabla_\bu \bF(\busg).
 \eeq
This motivates the introduction of
\beq
\label{e:cL0-inner}
\cLzi:= \mrL + \eps^2D^2\Delta_s,
\eeq
and the inner decomposition of $\cL$ as
\beq
\label{e:cL-reach2}
 \cL :=  \cLzi+ \eps \cLoi,
 \eeq
 where $\cLoi:=D^2(\kappa(s,z)\partial_z + \eps^2 D_{s,2})$.  On the complement of the reach of $\Gamma$ the operator $\cL$ is 
 an $\eps$-exponentially small perturbation of the \emph{constant-coefficient} operator
\beq
\label{e:cL0-def}
 \cLzo:=  \eps^2 D^2\Delta - \nabla_\bu \bF(\ba_0).
 \eeq
While $\cL^\dag\cL$ from (\ref{e:mbL-def}) is self-adjoint, its factor $\cL$ is not, indeed the spectrum of $\cL^\dag\cL$ is precisely the singular values of $\cL$. The spectrum of $\cL^\dag\cL$ is clearly real and non-negative. We define the space of meander modes which closely 
approximate an $N$-dimensional subspace of the tangent plane of $\cM_{K,\ell}(\bu_*)$, 
 \beq
 \label{e:Xn-def}
 X_N:= \left\{ \xi(z\eps)\psi_1(z) \theta_j(s) \bigl | \, j=1, \ldots N\right\},
 \eeq 
 where the translational eigenfunction $\psi_1=\partial_z \bu_*\in{\rm ker}(\mrL).$
 We also introduce $X_N^\dag$ which is the adjoint space obtained by replacing $\psi_1$ with the adjoint eigenfunction $\psi_1^\dag.$ 
We show that for $N$ sufficiently large that $\eps^2\beta_N^2$ is $O(1)$, then the operator $\mbL$ is coercive on $(X_N^\dag)^\perp$, 
uniformly in $\eps$, if the spectrum of $D^{-2}\mrL$ as an operator on $L^2(\R)$ has no strictly positive real spectrum.
\vskip 0.1in
\begin{thm}
\label{t:H2coer}
Let $\ba_0=0$ be a normally hyperbolic zero of $\bF$ and let $\bu_*$ be a solution of the 
freeway system (\ref{e:FW}) which is homoclinic to $\ba_0$, and let $\cM_{K,\ell}$ be the associated freeway manifold. 
Let the operators $\mrL=\mrL(\bu_*)$ and $\mbL=\mbL(\busg)$ be as given in (\ref{e:Ldef}) and (\ref{e:mbL-def}) respectively. 
If $\sigma_p(D^{-2}\mrL)\cap \R_+ = \{0\}$ with a simple 
kernel spanned by $\partial_z\bu_*$, then for the product $\ell K$ sufficiently small and for any fixed $\gamma_0>0,$ 
there exists $\eps_0, \mu>0$  such that for all $\eps\in(0,\eps_0)$ and all $\Gamma\in\cG_{K,\ell}$
   \beq
 \label{e:H2coer}
  \left\langle \mbL \bv,\bv\right\rangle_{L^2(\Omega)}\geq \mu \left(\eps^4 \|\Delta \bv\|_{L^2(\Omega)}^2+ \|\bv\|_{L^2(\Omega)}^2\right),
  \eeq 
  for all $\bv\in (X_N^\dag)^\perp$ where $N=N(\eps)$ is chosen to satisfy $\gamma_0<\beta_N^2 \eps^2<2\gamma_0$. 
\end{thm}

\begin{rmk}
\label{r:Pzero} The coercivity extends to any $O(\eps)$ regular perturbation of $\mbL$. Specifically the functionalization terms $\eps^p\bP$
in $\cF$ add an $O(\eps^p)$ regular perturbation to $\mbL$ that does not impact the coercivity. Neither the functionalization terms nor perturbations
to the form of $\bu_*$ will impact the coercivity, and specifically they cannot induce the pearling bifurcations to which the freeway manifolds of
the scalar FCH are susceptible.
\end{rmk}
\begin{proof}
\kpc{
From \eqref{e:mbL-def} and the bound \eqref{e:R} on $\mathbf R$, 
\beq\label{e:bL1}
\left\langle \mbL \bv, \bv\right\rangle_{L^2(\Omega)}\geq \left\langle \mbL_1\bv, \bv\right\rangle_{L^2(\Omega)} -C\eps\|\bv\|_{L^2(\Omega)}^2. 
\eeq}
The operator $\mbL_1$ admits distinct formulations when acting on functions supported in $\Gamma_{2\ell}$ and on
those supported in  $\Omega\setminus\Gamma_{2\ell}$. We decompose $\bv\in X_N^\perp$  as 
\beq
\bv=\bv_-+\bv_+:=\xi(\eps z) \bv+\oxi (\eps z) \bv,
\eeq
and writing $\mbL_1=\cL^\dag\cL$, we expand the left-hand side of \eqref{e:H2coer} as
\beq\label{e:inoutdecomp}
\left\langle \mbL_1 \bv, \bv\right\rangle_{L^2(\Omega)}=\| \cL\bv_-\|^2_{L^2(\Omega)} + 2\left\langle  \cL \bv_-, \cL\bv_+\right\rangle_{L^2(\Omega)} +\|\cL\bv_+\|^2_{L^2(\Omega)},
\eeq
denoting the summands on the right-hand side as the inner, mixed, and outer bilinear terms respectively,
we estimate them individually.

  \vskip 0.1in
\noindent \underline{Inner bilinear term}: From the inner formulation, (\ref{e:cL-reach2}), of $\mbL$ 
we focus on the leader-order operator $(\mathbb L_1)_{in,0}=\mathcal L_{in,0}^\dag \mathcal L_{in,0}$,
 and consider it as acting on functions defined on the abstract set $\bS_\infty:=\bS\times\R$ formed from the 
 unbounded whiskers.  This is not a subset of $\Omega$ as  the whiskers of length greater than $2\ell$ 
 generically intersect. We  first establish the coercivity of the operator $\mathbb L_{in,0}$ in the  $L^2(\bS_\infty)$ norm, defined as
\beq
\|\mathbf  f\|_{L^2(\bS_\infty)}^2=\int_{\mathbb R}\int_\bS |\mathbf f|^2 J_0(s) \, d s dz.
\eeq
As observed in (\ref{e:tJ-equiv}), for admissible interfaces $\Gamma\in\cG_{K,\ell}$ there exists $c>0$
such that 
\beq\label{e:L2equi}
\frac{\varepsilon}{c}\|\mathbf f\|^2_{L^2(\bS_\infty)} \leq \|\mathbf f\|_{L^2(\Gamma_{2\ell})}^2\leq c\varepsilon \|\mathbf f\|^2_{L^2(\bS_\infty)},
\eeq
for all $\bf\in L^2(\Gamma_{2\ell}).$ We also introduce  function space $H^2(\bS_\infty)$ with norm given by
 \beq
 \|\bf\|_{H^2(\bS_\infty)}^2=\int_{\mathbb R}\int_{\bS} \left(|\partial_z^2\mathbf f|^2+\varepsilon^4|\Delta_s\bf|^2+|\mathbf f|^2\right) J_0(s)\, dsdz.
 \eeq
 For functions supported in $\Gamma_{2\ell}$ the inner expression for $\eps^2\Delta$ given in
 (\ref{e:inner-Lap}) affords the estimates
 \beq\label{e:H2equi}
\frac{\eps}{c}\|\bf\|^2_{H^2(\bS_\infty)}\leq  \|\eps^2\Delta \bf\|^2_{L^2(\Gamma_{2\ell})}+
\|\bf\|_{L^2(\Gamma_{2\ell})}^2
\leq c\eps \|\bf\|^2_{H^2(\bS_\infty)}.
 \eeq

 \begin{lem}\label{lem:innercoersz}
 There exist $\mu_0>0$ such that
 for all $\eps\in(0,\eps_0)$ and all $\Gamma\in\cG_{K,\ell}$  
\beq
\|\cL_{in,0} \mathbf f\|_{L^2(\bS_\infty)}^2\geq  \mu_0 \,\|\bf\|_{H^2(\bS_\infty)}^2, 
\eeq
for all $\mathbf f\in (X_N^\dag)^\bot\cap H^2(\bS_\infty)$
\end{lem}
\begin{proof} We decompose 
\beq\label{e:fdecomp}
\bf=\bf^\bot+\bf^\parallel,
\eeq
where $\bf^\parallel$ the component of $\bf$ that lies in the $L^2(\bS_\infty)$-tangent plane to $X_N$, precisely,
\beq
\label{e:fparallel}
\bf^\parallel= \sum_{j=0}^n  \mathrm c_j \xi(\eps z)\psi_1(z)\theta_j(s), \qquad c_j=\left<\bf, \xi(\eps z)\psi_1(z)\theta_j(s)\right>_{L^2(\bS_\infty)}
\eeq
{\it Step 1: Control of $\bf^\parallel$ in $H^2(\bS_\infty)$.}
Since the $\{\theta_j\}_{j=0}^n$ are orthonormal in $L^2(\bS)$ and $|\eps^2 \beta_j^2|\leq 2\gamma_0$ for all $j=0, \ldots, n$, we have the estimate
\beq
\label{e:fpar1}
\|\mathbf f^\parallel\|_{H^2(\bS_\infty)}^2  \leq \left(1+\|\psi_1^\prime\|_{L^2(\R)}^2+4\gamma_0^2\right) \sum_{j=0}^n c_j^2.
\eeq
From the assumption $\bf\bot \xi\psi\theta_j$ in $L^2(\Omega)$ and the inner product representation \eqref{e:L2-inner}, we derive
\beq
0=\varepsilon \left< \bf, \xi\psi_1(z)\theta_j(s)\right>_{L^2(\bS_\infty)} + \left\langle  \bf,  \xi\psi_1(z)\theta_j(s) ( \tJ(s,z)-\varepsilon) \right\rangle_{L^2(\bS_\infty)},
\eeq
which allows us to rewrite the expression for $c_j$ in (\ref{e:fparallel}) as
\beq
c_j= - \left\langle  \bf \xi (\tJ-\eps)/\eps ,  \psi_1\theta_j\right\rangle_{L^2(\bS_\infty)}. 
\eeq
We  carry out the $z$ integral $\theta_j$ and express the remainder in the inner product \eqref{e:sinner}, 
\beq
c_j=\left\langle\tilde{\bf}(s), \theta_j \right\rangle_{s},
\eeq
where we have introduced
\beq
 \tilde{\bf}(s):= -\int_\R (\bf  \cdot \psi_1)\xi(\tJ-\eps)/\eps \, dz. 
\eeq
Using H\"older's inequality we bound
\beq
\|\tilde{\bf}\|_{L^2(\bS)}\leq \eps^{-1}\|(\tilde J-\eps)\psi_1\|_{L^2(\bS_\infty)} \|\bf\|_{L^2(\bS_\infty)}
\leq C \eps \|\bf\|_{L^2(\bS_\infty)}.
\eeq
For the last estimate we employed the second identify in (\ref{e:tJ-def}) together with the fact that
the higher order curvatures $K_j$ are uniformly bounded and and $\psi_1$ converges to zero exponentially in $|z|$
so that $|z|^k\psi_1$ is uniformly bounded in $L^2(\R)$ for $k=0,\ldots, d-1$. 
The functions $\{\theta_j(s)\}$ are orthonormal in the inner product \eqref{e:sinner}, so Plancherel's identity implies
\beq
\sum_{j=0}^n c_j^2\leq \|\tilde\bf\|_{L^2(\bS)}^2\leq C\eps^2 \|\bf\|_{L^2(\bS_\infty)}^2. 
\eeq 
for a different constant $C$. In particular, from  (\ref{e:L2equi}) and (\ref{e:fpar1}) it follows that
\beq\label{e:estfH2}
\|\bf^\parallel\|_{H^2(\bS_\infty)} \leq  C\eps \|\bf\|_{L^2(\bS_\infty)}\leq C\eps^{1/2} \|\bf\|_{L^2(\Gamma_{2\ell})}.
\eeq

\medskip
\noindent
{\it Step 2: $L^2(\bS_\infty)$-coercivity of $\cLzi$ on $(X^\dag_{N,\bS})^\bot$}.
The spaces $X_{N,\bS}$ and $X^\dag_{N,\bS}$ are the analogues of $X_N$ and $X_N^\dag$ 
in $L^2(\bS_\infty)$ obtained by dropping the $\xi$ cut-off in (\ref{e:Xn-def}). We establish the 
coercivity
\beq\label{e:incoerbot}
\|\cLzi \bf^\bot\|_{L^2(\bS_\infty)}\geq M^{-1} \|\bf^\bot\|_{H^2(\bS_\infty)},
\eeq
for $\bf^\bot\in (X^\dag_{N,\bS})^\bot$  through the equivalent estimate
\beq\label{e:incoerbot-inverse}
\|\cLzi^{-1}\bg^\bot \|_{H^2(\bS_\infty)}\leq M \|\bg^\bot\|_{L^2(\bS_\infty)},
\eeq
for all $\bg^\bot\in (X_{N,\bS}^\dag)^\bot$. Since $X^\dag_{N,\bS}$ is comprised of eigenspaces of $\cLzi^\dag$
it follows that the condition $\bg\perp X_{N,\bS}^\dag$ follows if and only if $\bf=\cLzi^{-1}\bg\perp X_{N,\bS}^\dag$.
In the remainder of step 2 we drop the $\bot$ superscript on $\bf$ and $\bg$,  
bounding $\bf\in H^2(\bS_\infty)$ where $\bf$ solves
\beq
\cLzi \bf=\bg,
\eeq  
subject to $\bg\in (X^\dag_{N,\bS})^\perp.$
We decompose $\bg$ and $\bf$ into their inner Fourier components
via the decomposition
 \beq
 \label{e:iFS}
 \bg = \sum\limits_{j=0}^\infty \bg_j(z)\theta_j(s),
 \eeq
 where the inner coefficients are given via formula
\beq
\label{e:Fourier}
 \bg_j(z):= \langle  \bg, \theta_j \rangle_s.
\eeq
using the inner product from \ref{e:sinner}. This yields the uncoupled sub-problems
 $$ (\mrL -D^2 \eps^2\beta_j^2) \bf_j = \bg_j.$$
 Since $\bu_*$ is a homoclinic and defined on all $\R$, the operators have natural extension to $L^2(\R).$ 
 We replace $\eps^2\beta_j^2$ with $k$ and define the family of operators
 $\{ \mrL_k\}_{k\geq 0},$  where $ \mrL_k:= (\mrL-D^2k).$  For each $\bh\in L^2(\R)$  we form the function
 $$G(k):= \| \mrL_k^{-1} \bh\|_{H^2(\R)}.$$
 By the spectral assumption of Theorem\,\ref{t:H2coer}, the operator $\mrL=\mrL_0$ has a simple eigenvalue at $0$, 
 which is removed by the projection off of $\psi_1^\dag$, while for all $k>0$, $\mrL_k$ is invertible from $L^2(\R)$ in $H^2(\R)$. 
 For $j=0,\ldots, N(\eps)$, corresponding to 
 $k=\eps^2\beta_j^2\in[0,\gamma_0]$,
 we have $\bg_j\in \{\bpsi_1^\dag\}^\perp$.  Consequently, for $k\in[0,\gamma_0]$ we consider  
 $\bh\in$[ker($\mrL_0^\dag)]^\perp = \{\bpsi_1^\dag\}^\perp$ so that $G$ is defined, 
 finite at $k=0$ and continuous in $k$ on $[0,\gamma_0]$. Since $[0,\gamma_0]$ is compact, 
 $G$ is uniformly bounded on this set for each $\bh$. From the uniform boundedness principal 
 we conclude that the  operators $\{\mrL_k^{-1}\}_{k\in[0,\gamma_0]}$ are uniformly norm 
 bounded from $\{\psi_1\}^\perp\subset L^2(\R)$ to $H^2(\R)$. For $k\geq \gamma_0$ we consider 
 $\bh\in L^2(\R)$, and observe that $G$ is finite for each 
 $k$, continuous in $k$, and converges to zero as $k\to+\infty.$ For each $\bh\in L^2(\R)$, we 
 deduce that $\sup_{k\geq \gamma_0} G(k)<\infty$, and from the uniform boundedness principal  
 we conclude that the  operators $\{\mrL_k^{-1}\}_{k\geq\gamma_0}$ are uniformly norm 
 bounded from $L^2(\R)$ to $H^2(\R)$. These bounds are independent of $\eps\in(0,\eps_0)$ 
 as the operators are independent of $\eps.$  Since
 $$ \cLzi^{-1} \, \bg =  \sum\limits_{j=0}^\infty \theta_j \mrL_{\eps^2\beta_j^2}^{-1} \bg_j,$$
 and since the Laplace-Beltrami eigenmodes are orthonormal in the $L^2(\bS_\infty)$ inner product, we deduce the existence of   $M>0$ such that 
 (\ref{e:incoerbot-inverse}) holds, and hence (\ref{e:incoerbot}) follows.
Since $\xi(z\eps)\psi_1^\dag$ and $\psi_1^\dag$ are $\eps$-exponentially close, the
coercivity in $X_N^\bot$ follows with an $\eps-$exponentially small modification to $M^{-1}>0$.

 \medskip
 \noindent
 {\it Step 3: Coercivity of $\cLzi$ on $X_N^\bot$ in $L^2(\Omega)$.}  The decomposition \eqref{e:fdecomp} provides the lower bound
 \beq
 \|\cLzi \bf\|_{L^2(\bS_\infty)}^2\geq \frac{1}{2}\|\cLzi \bf^\bot\|_{L^2(\bS_\infty)}^2-\|\cLzi \bf^\parallel\|_{L^2(\bS_\infty)}^2
 \eeq
 By the coercivity of $\cLzi$ in {\it Step 2}, the first positive term has a $H^2(\bS_\infty)$ lower bound while the negative term can be bounded by $H^2(\bS_\infty)$-norm of $\mathbf f^\parallel$. More precisely, 
 \beq
 \|\cLzi \mathbf f\|_{L^2(\bS_\infty)}^2\geq \frac{M^{-1}}{2} \|\mathbf f^\bot\|_{H^2(\bS_\infty)}^2-C\|\mathbf f^\parallel\|_{H^2(\bS_\infty)}^2.
 \eeq
On the other hand, a second application of \eqref{e:fdecomp} yields
 \beq
 \|\bf^\bot\|^2_{H^2(\bS_\infty)}\geq  \frac{1}{2}\|\bf\|_{H^2(\bS_\infty)}^2 -\|\bf^\parallel\|_{H^2(\bS_\infty)}^2,
 \eeq
which combined with the previous inequality implies 
\beq
 \|\cLzi \bf\|_{L^2(\bS_\infty)}^2\geq \frac{M^{-1}}{4} \|\bf\|_{H^2(\bS_\infty)}^2-
 C\|\bf^\parallel\|_{H^2(\bS_\infty)}^2
\eeq
for a different constant $C$. The coercivity Lemma follows from \eqref{e:estfH2} by replacing the $H^2(\bS_\infty)$  bound of $\bf^\bot$ with the $L^2(\bS_\infty)$ bound of $\bf$. 
 \end{proof}
 To derive a lower bound on the inner bilinear form we account for the lower order terms.
Since the support of  $\mathbf v_-$ lies in $\Gamma_{2\ell}$ we may apply the decomposition \eqref{e:cL-reach2}
  \beq
\| \mathcal L \mathbf v_-\|^2_{L^2(\Omega)}=  \|(\mathcal L_{in,0}+\eps\mathcal L_{in,1})\mathbf v_-\|^2_{L^2(\Omega)},
  \eeq
expand the quadratic form, and apply Young's inequality to the sign-undeterminate term, 
  \beq\label{e:incoerest1}
  \|\cL \bv_-\|_{L^2(\Omega)}^2 \geq \frac{1}{2}\|\cLzi \bv_-\|_{L^2(\Omega)}^2-\eps^2\|\cLoi \bv_-\|_{L^2(\Omega)}^2.
  \eeq
 Since the support of $\bv_-$ lies inside $\Gamma_{2\ell}$  its $L^2(\Omega)$ and $L^2(\bS_\infty)$ norms satisfy the $\eps$-equivalency of (\ref{e:H2equi}). 
 Applying the $L^2$-coercivity of Lemma \ref{lem:innercoersz}, we arrive at the lower bound
  \beq\label{e:incoerest2}
  \begin{aligned}
  \frac{1}{2}\|\cLzi \bv_-\|_{L^2(\Omega)}^2&\geq \frac{\eps}{4}\|\cLzi \bv_-\|_{L^2(\bS_\infty)}^2\geq \frac{\eps  \mu_0}{4}\|\bv_-\|_{H^2(\bS_\infty)}^2 \geq \frac{ \mu_0}{8}\left(\|\bv_-\|_{L^2(\Omega)}^2+\|\eps^2\Delta \bv_-\|_{L^2(\Omega)}^2\right).
  \end{aligned}
  \eeq 
  In addition, in light of the definition (\ref{e:cL-reach2}) of $\cLoi$,
 \beq\label{e:incoerest3}
 \begin{aligned}
  \eps^2\|\cLoi \bv_-\|_{L^2(\Omega)}^2&
  \leq 2\eps^2 \|\kappa\|_{L^\infty(\Gamma_{2\ell})}^2 \|\partial_z \bv_-\|_{L^2(\Omega)}^2  +2\|\eps z\|_{L^\infty(\Gamma_{2\ell})}^2\| \varepsilon^2 D_{s,2} \mathbf v_-\|_{L^2(\Omega)}^2,
 \\
 &\leq 2K\eps^2 \|\partial_z \bv_-\|_{L^2(\Omega)}^2 +2(K\ell)^2 \|\eps^2 D_{s,2} \bv_-\|^2_{L^2(\Omega)},\\
 &\leq \eps C(\eps^2+(K\ell)^2) \|\bv_-\|^2_{H^2(\bS_\infty)}\leq C(\eps^2+(K\ell)^2)
 \left(\|\bv_-\|_{L^2(\Omega)}^2+\|\eps^2\Delta \bv_-\|_{L^2(\Omega)}^2\right).
 \end{aligned}
 \eeq  
Here we used the  Gargliardo-Nirenberg embedding inequality and the $\eps$-equivalence of the $H^2(\Omega)$ and $H^2(\bS_\infty)$ norms.  Combining \eqref{e:incoerest1}-\eqref{e:incoerest3} and taking
$\eps$ and $K\ell$ sufficiently small yields the inner coercivity:
 \beq\label{e:in}
 \|\cL \bv_-\|_{L^2(\Omega)}^2\geq \frac{ \mu_0}{16}
\left(\|\bv_-\|_{L^2(\Omega)}^2+\|\eps^2\Delta \bv_-\|_{L^2(\Omega)}^2\right).
 \eeq

 \vskip 0.1in
 \noindent \underline{Outer bilinear term}: 
 Recalling the leading order, constant coefficient outer form $\cLzo$ of $\cL$, given in (\ref{e:cL0-def}), 
 we define the outer operator
 $\mbLzo:=\cLzo^\dag \cLzo$. 
 \begin{lem}\label{lem:outcoer}
For any $\bf\in H^2(\Omega)$, there exist $ \mu_0>0$ such that
\beq
\label{e:bmLzo-coercivity}
\langle \mbLzo\bf,\bf\rangle_{L^2(\Omega)}\geq \mu_0 \|\bf\|^2_{L^2(\Omega)}.
\eeq
 \end{lem}
\begin{proof}
By assumption $\ba=0$ is a normally hyperbolic zero of $\bF$, in particular from 
Lemma\,\ref{l:Ham} we know that $\sigma(D^{-2}\nabla_\bu F(0))$ has no eigenvalues in $(-\infty,0]$. 
Form the form of $\mbLzo$ we have ker($\mbLzo)=$ ker$(\cLzo)$, and since $\Omega$ is a rectangular box subject to periodic boundary conditions, the kernel of 
$\cLzo$ is comprised of functions of the form $e^{ik\cdot x}\bU$, where $\bU\in\R^d$ is a constant vector. This function lies in the kernel only if 
$$-\eps^2|k|^2D^2\bU - \nabla_\bu\bF(0)\bU=0,$$
or equivalently 
$$-\eps^2 |k|^2\in \sigma(D^{-2}\nabla_\bu \bF(0)).$$ 
However this is precisely the condition precluded by the assumption that $\ba=0$ is normally hyperbolic, thus we deduce
that $\mbLzo$ has no kernel and is invertible. However we can make a stronger statement. The spectrum of $\mbLzo$ is discrete, but lies on the curves of essential spectrum defined 
as the finite family of dispersion relations $\lambda=\lambda(k)$ for which the $d\times d$ matrix 
$$M(k,\lambda):=-\eps^2|k|^2D^2-\nabla_\bu \bF(0) -\lambda,$$
has a kernel, see chapter 3 of \cite{KP-book}. The assumption of normal hyperbolicity implies that none of the dispersion relation curves pass through the origin. Indeed, rescaling $k\in\R^d$, the dispersion relation curves can be made independent of $\eps$, and tend to $-\infty$ as $|k|\to\infty$. This implies that the curves lie a finite distance $\mu_0>0$
to the origin, which is wholly independent of $\eps.$ Since $\mbLzo$ is self-adjoint and non-negative, this spectral 
bound provides the coercivity estimate (\ref{e:bmLzo-coercivity}).
 \end{proof}
 
 We expand the outer bilinear term as
 \beq
 \|\cL \bv_+\|_{L^2(\Omega)}^2=\int_\Omega |\cLzo \bv_++\left(\nabla_{\bu} \bF(0)-\nabla_{\bu}\bF(\busg)\right) 
  \bv_+|^2  dx,
 \eeq
which from Young's inequality enjoys the lower bound
 \beq
 \|\cL\bv_+\|_{L^2(\Omega)}^2 \geq \frac{1}{2}\|\cLzo \bv_+\|_{L^2(\Omega)}^2- 2\|\left(\nabla_{\bu} \bF(0)-
 \nabla_{\bu}\bF(\busg)\right) \bv_+\|_{L^2(\Omega)}^2.
 \eeq
 Since the support of $\bv_+$ lies in  $\Omega\setminus \Gamma_{\ell}$, we have the bound
 \beq
2\|\left(\nabla_{\bu} \bF(0)-\nabla_{\bu}\bF(\busg)\right) \bv_+\|^2_{L^2(\Omega)} \leq  
2\|\nabla_{\bu} \bF(0)-\nabla_{\bu}\bF(\busg)\|_{L^\infty(\Omega\setminus \Gamma_{\ell})}^2 \|\bv_+\|_{L^2(\Omega)}^2.
 \eeq
 The homoclinic connection $\mathbf u_*$ converges to $0$ with an exponential rate as $z$ goes to 
 infinity and the function $\bF$ is smooth, so the $L^\infty$-norm of the difference on 
 $\Omega\setminus \Gamma_{2\ell}$ is $\eps$-exponential small.  This establishes the bound
 \beq\label{e:out}
 \|\cL\bv_+\|_{L^2(\Omega)}^2 \geq \frac{\mu_0}{2}\|\bv_+\|_{L^2(\Omega)}^2- C e^{-\ell/\eps} \|\bv_+\|_{L^2(\Omega)}^2\geq \frac{\mu_0}{4}\|\bv_+\|_{L^2(\Omega)}^2.
 \eeq
 Here the constant $C$ depends only on $\bF$ and the decay rates of $\bu_*$ in $z$.
 \vskip 0.1in
 \noindent \underline{Mixed bilinear terms}: The support of $\bv_+\bv_-$ is contained in the 
 overlap region $\Gamma_{2\ell}\setminus \Gamma_{\ell}$. On this set the difference 
 $\nabla_{\bu} \bF(\busg)-\nabla_{\bu} \bF(0)$ is $\eps$-exponentially small, and we use the
 outer expansion of $\cL$ to obtain
 \beq
 \label{e:BL-mixed}
 \begin{aligned}
 \left\langle \cL \bv_-, \cL\bv_+\right\rangle_{L^2(\Omega)} &\geq \left\langle  \cLzo\bv_-, \mathcal L_{out, 0} \mathbf v_+ \right\rangle_{L^2(\Omega)}- Ce^{-\ell/\eps} \left(\left\langle |\bv_-|, |\cL\bv_+|\right\rangle_{L^2(\Omega)} +\left\langle |\cL\bv_-|, |\bv_+|\right\rangle_{L^2(\Omega)}\right).
 \end{aligned}
 \eeq
 Applying H\"older's inequality, the negative term on the right hand side can be bounded  by 
 \beq
 Ce^{-\ell/\eps}  \left(\|\bv_-\|_{L^2(\Omega)}^2+\|\bv_+\|_{L^2(\Omega)}^2 + \|\cL\bv_-\|_{L^2(\Omega)}^2+
 \|\cL \bv_+\|_{L^2(\Omega)}^2\right).
 \eeq
Since  $\bv_-=\xi(\eps z) v(x)$, its support is contained with $\Gamma_{2\ell}$ and 
we may use the inner expression for the Laplacian to obtain a lower bound on the first 
term on the right-hand side of (\ref{e:BL-mixed}). Moreover $\xi$ is slowly varying in $x$ and
independent of $s$. With these observations we obtain the expansion
 \beq\label{e:inLapv}
 \begin{aligned}
 \eps^2 \Delta \bv_-= & \xi\, \eps^2 \Delta \bv+ 2\eps \xi' \partial_z \bv+\eps^2 \xi'' \bv+\eps^2 \kappa\xi' \bv, \\
 \end{aligned}
 \eeq
with a similar expansion for $\bv_+$ with $\xi$ replaced with $\oxi.$ Since $\xi$ and its derivatives are uniformly bounded, independent of $\eps$, we obtain 
 \beq
 \label{e:inLapv2}
 \begin{aligned}
&  \left\langle  \cLzo \bv_-, \cLzo \bv_+ \right\rangle_{L^2(\Omega)}\\
 &\quad \geq\left\langle  \xi \cLzo \bv, \oxi \cLzo \bv \right\rangle_{L^2(\Omega)}-C\eps \left(\|\eps^2 \Delta \bv\|_{L^2(\Omega)}^2+\|\partial_z \bv
 \|_{L^2(\Gamma_{2\ell}\setminus\Gamma_\ell)}^2+\|\bv\|_{L^2(\Omega )}^2\right). \end{aligned}
 \eeq
The first term on the right-hand side of (\ref{e:inLapv2}) is positive since the product $\xi\oxi$ is non-negative.  
Using the $\eps$-equivalence  of $L^2(\Omega)$ and $L^2(\bS_\infty)$ norms from \eqref{e:L2equi}, and standard embedding inequalities,
we obtain 
 \beq
\|\partial_z \bv\|_{L^2(\Gamma_{2\ell}\setminus\Gamma_\ell)}^2
\leq   \eps \|\partial_z\bv\chi_{\{x\in \Gamma_{2\ell}\setminus \Gamma_{\ell}\}}\|^2_{L^2(\bS_\infty)}
\leq   C\eps \| \bv\|_{H^2(\bS_\infty)}.
 \eeq 
Finally, from the $H^2$-norm $\eps$-equivalence given in \eqref{e:H2equi}  we deduce
  \beq\label{e:inout}
 \begin{aligned}
&\quad  \left\langle  \cLzo \bv_-, \cLzo \bv_+ \right\rangle_{L^2(\Omega)}\geq 
 -C\eps \left(\|\eps^2 \Delta \bv\|_{L^2(\Omega)}^2+\|\bv\|_{L^2(\Omega)}^2\right).
 \end{aligned}
 \eeq

Combining the lower bounds on the inner, \eqref{e:in}, outer \eqref{e:out}, and mixed \eqref{e:inout} bilinear forms, 
with  the decomposition  \eqref{e:inoutdecomp} { and inequality \eqref{e:bL1}}, we obtain the existence of $\tilde \mu>0$, independent of
$\eps$ for which
 \beq
\left<\mbL \mathbf v, \mathbf v\right>_{L^2(\Omega)} \geq \tilde \mu \|\mathbf v\| _{L^2(\Omega)}^2-C\varepsilon \|\varepsilon^2 \Delta \mathbf v\|_{L^2(\Omega)}^2.
 \eeq
 From the form of $\mbL$, elliptic regularity theory affords the existence of $\gamma>0$, 
 independent of $\eps>0$, such that 
 $$ \langle (\mbL +\gamma) \bv, \bv\rangle_{L^2(\Omega)} \geq  \|\eps^2\Delta \bv\|_{L^2(\Omega)}^2,$$
 for all $\bv\in H^2(\Omega).$
 Then for any $t\in (0,1)$ we may interpolate
 \beq 
\begin{aligned}
  \langle \mbL \bv, \bv\rangle_{L^2(\Omega)} &  \geq t  \langle (\mbL +\gamma) \bv, \bv\rangle_{L^2(\Omega)} + 
                      \left((1-t)\mu_0-t\gamma\right) \|\bv\|_{L^2(\Omega)}^2-C\eps(1-t)\|\eps^2\Delta \bv\|_{L^2(\Omega)}^2, \\
              & \geq \left(t -C(1-t) \eps\right)  \|\eps^2\Delta \bv \|_{L^2(\Omega)}^2 + \left(\tilde{mu}-t(\tilde{\mu}+\gamma)\right) \|\bv\|_{L^2(\Omega)}^2.
\end{aligned}
\eeq
The choice $t:= \frac{\tilde{\mu} +C\varepsilon}{1+\tilde{\mu}+\gamma +C\varepsilon}$ yields the estimate  \eqref{e:H2coer} with
\beq
\mu=\frac{\tilde \mu}{2(1+\tilde \mu+\gamma)}>0,
\eeq       
independent of $\eps$. 
\end{proof}

\subsection{Normal coercivity of singular homoclinic freeway manifolds}
The results of section\,\ref{s:FW_singpert_existence}  provide constructive conditions for the existence of homoclinic freeway connections in
the freeway system (\ref{e:FW}) with $\bF$ as in (\ref{e:FW_singpert_Fdelta}). Section\,\ref{s:FM} constructs the corresponding freeway manifold, which from Proposition\,
\ref{p:FW_manifold} is comprised of low energy functions.  Theorem\,\ref{t:H2coer} of section\,\ref{s:NC} equates the 
normal coercivity of the associated freeway manifold to a spectral condition on the linearization (\ref{e:Ldef}) 
of the one-dimensional freeway system at the underlying homoclinic freeway connection, $\bu_*$. For the singularly
perturbed systems of section \ref{s:FW_singpert_existence}, the spectral problem has been been analyzed in 
detail \cite{DoelmanVeerman.2015}. In particular the stability hypothesis of Theorem \ref{t:H2coer} can be related to simple geometric 
conditions arising in the construction of the slow-fast homoclinic freeway connections.

Assume the framework of section\,\ref{s:FW_singpert_existence} and that the function $\rho$, 
defined in \eqref{e:FW_singpert_rho}, has a simple root $s_*>0$.
Let $\bu_*$ be the associated slow-fast homoclinic freeway connection. 
Then, under the assumption that
\begin{equation}\label{e:intF12_nonzero}
\int_\R \!\bF_{12}(s_*,u_{2,h}(\zeta;s_*);0)\,\text{d}\zeta \neq 0,
\end{equation}
\cite[Corollary 5.10 and eq. (5.16)]{DoelmanVeerman.2015} imply that the kernel of $\mrL$, and hence that of 
$D^{-2} \mrL$, is simple and spanned by the translational eigenmode $\partial_z \bu_*$. 
To apply Theorem \ref{t:H2coer} it remains to verify that $\sigma_p(D^{-2}\mrL)$ has no
strictly positive elements. To this end it is convenient to consider the point spectrum of the operator pencil
$D^{-2}\left(\mrL - \lambda\right)$ for $\lambda\in\C$. For any $k \in \sigma_p(D^{-2}\left(\mrL-\lambda\right))$ there exists a 
solution $\psi\in L^2(\R)$ to the eigenvalue problem
\begin{equation}\label{e:FW_singpert_evp}
\mrL \psi = \begin{pmatrix} \lambda + k & 0 \\ 0 & \lambda + \delta^2 k \end{pmatrix}\psi.
\end{equation}
This eigenvalue problem has precisely the same structure as that in \cite[eq (3.2)]{DoelmanVeerman.2015}, modulo the replacement of 
`$\lambda$' by `$\lambda+k$' in the first component and `$\lambda$' by the asymptotically close value `$\lambda + \delta^2 k$' in the second 
component. All the assumptions of \cite{DoelmanVeerman.2015} hold for this extended problem, as do each of the steps of the subsequent 
analysis. Indeed, the set-up of this situation is has similarities to the stability analysis of homoclinic stripes in singularly perturbed
reaction-diffusion systems conducted in \cite{SD17} with the exception that the case $k>0$ was not considered therein. 
It follows from the prior analysis that there exists an extended analytic Evans function $\mathcal{D}(\lambda,k,\delta)$ whose roots coincide with 
the point spectrum of the operator pencil $D^{-2}(L-\lambda)$, including multiplicity. 
Moreover, there exists an analytic fast transmission function $t_{f,+}$ and a 
meromorphic slow transmission function $t_{s,+}$ such that the extended Evans function admits the slow-fast decomposition
\begin{equation}\label{e:FW_singpert_Evansf}
\mathcal{D}(\lambda,k,\delta) = 4 \delta \,t_{f,+}(\lambda+\delta^2 k,\delta) \,t_{s,+}(\lambda,k,\delta) \sqrt{\partial_{u_1}\bF_2 (0,0;\delta)+\lambda+\delta^2 k} \sqrt{\bF_{11}^{\;\prime}(0;\delta)+\lambda+k},
\end{equation}
see \cite[eq. (4.4)]{DoelmanVeerman.2015}. This Evans function decomposition, which follows from the strong structural similarity between the eigenvalue problem \eqref{e:FW_singpert_evp} and the stability problem studied in \cite{DoelmanVeerman.2015}, allows us to prove the following Theorem.

\begin{thm}\label{t:FW_singpert_Evansf}
Suppose that the vector field $\bF(\bu;\delta)$ is as given in \eqref{e:FW_singpert_Fdelta}, the assumptions of section \ref{s:FW_singpert_existence} hold, and $s_*$ is a simple root of $\rho$ given in (\ref{e:FW_singpert_rho}). Let $\bu_*$ be the associated freeway homoclinic connection to $\ba=0$. Suppose that, in addition,
\begin{align}
 \rho'(s_*) >& 0,\label{e:t:FW_singpert_Evansf_cond1}\\
 \int_\R \!\bF_{12}(s_*,u_{2,h}(\zeta;s_*);0)\,\mathrm{d}\zeta <& 0,\label{e:t:FW_singpert_Evansf_cond2}
\end{align}
where $u_{2,h}(\zeta;u_1)$ is as defined in Assumption \ref{a:A2}. 
Then, the set $\sigma_p \left(D^{-2}\mrL\right) \cap \R_+$ consists of precisely one simple eigenvalue at the origin.
\end{thm}

\begin{proof}
The assumption \eqref{e:t:FW_singpert_Evansf_cond2}, together with the fact that $s_*$ is a simple root of $\rho$, guarantees the simplicity of the eigenvalue at zero. 
Hence, it is sufficient to show that the Evans function  $\mathcal{D}(0,k,\delta)$ \eqref{e:FW_singpert_Evansf} has no zeroes for $k > 0$ and $\delta$ sufficiently small. 
By \cite[Lemma 4.3]{DoelmanVeerman.2015}, the roots 
of the fast transmission function $t_{f,+}(\lambda,\delta)$ are to leading order in $\delta$ given by the eigenvalues of the fast Sturm-Liouville operator $L_f := \partial_{\zeta}^2 - \partial_{u_2}\bF_2(s_*,u_{2,h}(\zeta;s_*))$.
Since $L_f$ is the linearization of \eqref{e:FW_singpert_fastsub} at the planar homoclinic $u_{2,h}$, it has a kernel associated to the translational invariance of the planar system. 
This kernel is isolated and simple by the Sturm separation theorem. 
Hence, by the inverse function theorem, $t_{f,+}(\delta^2 k,\delta) \neq 0$ for sufficiently small $\delta$.

By \cite[Theorem 4.4]{DoelmanVeerman.2015}, we can express the slow transmission function $t_{s,+}(\lambda,k,\delta)$ to leading order in $\delta$ as
\begin{equation}\label{e:thm3.8_pf_tsk}
 t_{s,+}(\lambda,k,0) = -\frac{B_-^2(\lambda+k)}{\Lambda_s(k)}\left[\frac{B_-'(\lambda+k)}{B_-(\lambda+k)}-\frac{B_-'(\lambda)}{B_-(\lambda)}-\frac{\Lambda_s(0)}{B_-^2(\lambda)}t_{s,+}(\lambda,0,0)\right],
\end{equation}
where $\Lambda_s(\lambda) = \sqrt{\bF_{11}'(0;\delta) + \lambda}>0$ (cf. \cite[eq. (3.8)]{DoelmanVeerman.2015}) and $B_-(\lambda)$, $B_-'(\lambda)$ are as defined in \cite[Theorem 4.4]{DoelmanVeerman.2015}. 
By \cite[Lemma 5.9]{DoelmanVeerman.2015}, for $\lambda=0$, we can write 
\begin{equation}\label{e:thm3.8_pf_ts0}
t_{s,+}(0,0,0) = -c_s\, \rho'(s_*)\int_\R\!\bF_{12}(s_*,u_{2,h}(\zeta;s_*);0)\,\text{d}\zeta,
\end{equation}
with $c_s>0$, using \cite[eq. (2.9)]{DoelmanVeerman.2015}. 
From \cite[Lemma 5.6]{DoelmanVeerman.2015} we know that $B_-(\lambda) \neq 0$ for all $\lambda \geq 0$ if and only if $y_*>0$, where 
\begin{equation}
\text{sgn }y_* = -\text{sgn }\int_\R\!\bF_{12}(s_*,u_{2,h}(\zeta;s_*);0)\,\text{d}\zeta,
\end{equation}
see \cite[Lemma 2.2]{DoelmanVeerman.2015}. 
We employ a Pr{\"u}fer transformation \cite[eq. (5.5)]{DoelmanVeerman.2015} to write 
\begin{equation}
\frac{B_-'(\lambda+k)}{B_-(\lambda+k)} = \tan\,\theta(\lambda+k),
\end{equation}
where $\theta:\R\mapsto\R$.
From the statement of \cite[Lemma 5.4]{DoelmanVeerman.2015} we deduce 
the strict monotonicity of  $\theta$, and conclude
\begin{equation}\label{e:thm3.8_pf_Bk0}
\frac{B_-'(k)}{B_-(k)}<\frac{B_-'(0)}{B_-(0)}
\end{equation}
for all $k > 0$. 
Combining \eqref{e:thm3.8_pf_ts0} with the assumptions \eqref{e:t:FW_singpert_Evansf_cond1} and \eqref{e:t:FW_singpert_Evansf_cond2} implies that
\begin{equation}
\frac{\Lambda_s(0)}{B_-^2(\lambda)}t_{s,+}(\lambda,0,0)>0,
\end{equation}
which can be taken together with \eqref{e:thm3.8_pf_Bk0} to conclude that factor within the square brackets in \eqref{e:thm3.8_pf_tsk} is negative, while the prefactor $B_-(\lambda+k)$  is finite and never zero. 
We deduce that  $t_{s,+}(0,k,0)> 0$ for all $k >0$. 
The non-vanishing of the Evans function \eqref{e:FW_singpert_Evansf} for $\lambda=0$ now follows from \cite[Corollary 4.2]{DoelmanVeerman.2015}.
\end{proof}

\begin{cor}
\label{c:RS_SPS}
 Suppose that the assumptions of Theorems  \ref{t:H2coer} and  \ref{t:FW_singpert_Evansf} are met. Then, there exists
 $\delta_0$, $\cG_0$>0 for which each $\delta\in (0,\delta_0)$ and each $ K, \ell>0$ satisfying $K\ell<\cG_0$ yield an $\eps_0>0$
 and a $\mu>0$ such that the freeway homoclinic connection $\bu_*$ of \eqref{e:FW} corresponding to the system presented in \eqref{e:FW_singpert_Fdelta} generates a normally coercive manifold $\cM_{K,\ell}(\bu_*)$, satisfying (\ref{e:H2coer}) with coercivity
 constant $\mu$ for all $\mbL=\mbL_\Gamma$ with $\Gamma\in \cG_{K,\ell}.$
\end{cor}

The PCB system presented in section\,\ref{s:PCB} prescribes a take-off curve and an unstable slow manifold, as depicted in Figure \ref{fig:cholesterol}. When the take-off curve crosses the unstable manifold from above, as it does at $u_1=s_*$, then $\rho'(s_*)>0$
and Corollary\,\ref{c:RS_SPS} holds. In particular the freeway manifold generated by $\bu_*$ is normally coercive in the sense of Theorem\,\ref{t:H2coer}. 

\section{Freeway to Toll-Road Bifurcations}
\label{s:Barriers}
Minimizers of the reduced free energy \eqref{eq:F1-def} solve the toll-road system \eqref{e:FO}. In this section we consider bifurcations within the freeway system \eqref{e:FW} that induce changes in solution type within the larger toll-road system. We insert a parameter, $\mu$, 
within the vector field $\bF=\bF(\cdot;\mu)$. When written as pair of second order systems, the toll-road system  \eqref{e:FO4n} has the
equivalent formulation
\bsub\label{e:TR_system}
\begin{align}
 D^2\,\bu_{zz} &= \bF(\bu;\mu)+\bv,\label{eq:uv_system_uzz}\\
 D^2\,\bv_{zz} &= \nabla_\bu \bF(\bu;\mu)^\dag \bv.\label{eq:uv_system_vzz}
\end{align}
\esub
The freeway solutions satisfy \eqref{e:TR_system} with $\bv = 0$.\\

In this section, we assume that for $\mu \geq 0$ the toll-road system \eqref{e:TR_system} admits a one-parameter pair of 
freeway connections $(\bu_{\pm}(\mu),0)^t$ between the fixed zeros $\ba_i$ and $\ba_j$ of $\bF$. Moreover, we assume that the two branches 
merge at $\mu=0$ through a saddle-node bifurcation, with $\bu_+(0) = \bu_-(0) := \bu_0$.  We shift the origin $(\bu,\bv)\mapsto (\bu_0+\bu,\bv)$ and expand \eqref{e:TR_system} around the connection $(\bu_0,0)^t$ at the saddle-node bifurcation $\mu=0$.
This results in the formulation
\begin{equation}
\hmbL \bpm \bu \\ \bv \epm + \bR(\bu,\bv;\mu) = 0, \label{eq:fullsystem_u0}
\end{equation}
where we have introduced
\begin{align}
 \hmbL &= \bpm
  \mrL & - I \\ 0 & \mrL^\dag
 \epm, &
 \bR(\bu,\bv;\mu) &= - \bpm
 \bF(\bu_0+\bu;\mu) - \bF(\bu_0;0) - \nabla_\bu \bF(\bu_0;0)\, \bu \\ \left[ \nabla_\bu \bF(\bu_0+\bu;\mu)^\dag - \nabla_\bu \bF(\bu_0;0)^\dag \right] \bv
 \epm.\label{eq:def_bbL_R}
\end{align}
As before $\mrL$, defined in \eqref{e:Ldef}, is the linearization of \eqref{e:FW} at $\bu_0$ with $\mu=0$. The nonlinear remainder term $\bR(\bu,\bv;\mu)$ \eqref{eq:def_bbL_R} can be expanded for small $(\bu,\bv)^t$ and small $\mu$, yielding
\begin{align}
\bR(\bu,\bv;\mu) &= -\bpm
\mu \partial_\mu \bF + \frac{1}{2}(\nabla_\bu^2 \bF)(\bu,\bu) + (\partial_\mu \nabla_\bu \bF)\,\mu \bu \\
(\nabla_\bu^2 \bF^\dag)(\bu,\bv) + (\partial_\mu \nabla_\bu \bF^\dag)\, \mu \bv
\epm
+ O\left((\|(\bu,\bv)^t\| + |\mu|)^3\right),
\end{align}
where we assume for simplicity that $\bF(\bu;\mu)$ depends linearly on $\mu$, i.e. 
$\partial_\mu^2 \bF(\bu;\mu) \equiv 0.$

We assume that the saddle-node bifurcation at $\mu=0$ is non-degenerate. Due to translational invariance 
$\psi_1:=\partial_z\bu_0\in \text{ker} (\mrL)$ and the saddle-node bifurcation yields another central direction. Specifically
\begin{equation}
 \text{ker}(\mrL) = \left\{ \bpsi_0,\bpsi_1 \right\},
\end{equation}
with
\begin{equation}\label{e:SN_eigenvector}
 \psi_0 = \lim_{\mu \downarrow 0} \frac{1}{2\sqrt{\mu}} \left(\bu_+(\mu) - \bu_-(\mu)\right).
\end{equation}
From the structure of $\hmbL$ and the Fredholm alternative we deduce that
\begin{equation}
\text{ker}(\hmbL) = \left\{ \bpm \psi_0 \\ 0 \epm, \bpm \psi_1 \\ 0 \epm \right\}, \label{eq:kerL}
\end{equation}
where $\psi_0$ and $\psi_1$ are even and odd, respectively, about $z=0.$ We introduce  $\text{ker}\left(\mrL^\dag\right) = \left\{\bpsi_0^\dag,\bpsi_1^\dag \right\}$, with $\psi_0^\dag$ and $\psi_1^\dag$ also even resp. odd about $z=0$. The spectral projections onto $\bpsi_j$ and $\bpsi_j^\dag$ are given by
\begin{equation}
 \Pi_j \bu = \frac{\langle \bu,\bpsi_j^\dag \rangle}{\langle \bpsi_j,\bpsi_j^\dag\rangle} \bpsi_j\quad\text{and}\quad \Pi_j^\dag \bu = \frac{\langle \bu,\bpsi_j \rangle}{\langle \bpsi_j^\dag,\bpsi_j\rangle} \bpsi_j^\dag,\quad j=0 \text{ or }1,
\end{equation}
with complementary projections $\tPi_j = I - \Pi_j$ and $\tPi_j^\dag = I - \Pi_j^\dag$. 
 
\subsection{Normal form expansion}\label{s:ex:nf_exp}
We perform a normal form expansion in \eqref{eq:fullsystem_u0}. We write the perturbative term $(\bu,\bv)^t$ in the form
\begin{equation}\label{eq:normform}
\bpm \bu \\ \bv \epm = \rho \bpm
\bpsi_0 \\ 0 \epm + \bpm \Phi(\rho\,\bpsi_0,\mu) \\ \Psi(\rho\,\bpsi_0,\mu) \epm,
\end{equation}
where the nonlinear functions $\Phi,\Psi$ are expanded as
\begin{equation}\label{eq:nf_expansion}
\Phi(\rho\,\psi_0,\mu) = \mu\,\Phi_{01} + \!\!\!\sum_{2 \leq p+q \leq N} \rho^p\mu^q  \Phi_{pq}(\psi_0,\psi_0,\ldots,\psi_0) + O\left((\rho+\mu)^{N+1}\right)
\end{equation}
for small $\rho$ and $\mu$; here, $\Phi_{pq}$ is a $q$-linear map. $\Psi$ is expanded analogously.

\begin{rmk}
While the translational invariance of \eqref{e:FO} introduces a central direction through the $z$-derivative of $\bu_0$, the same translational invariance precludes $\psi_1 = \partial_z \bu_0$ to play a direct role in the normal form expansion \eqref{eq:normform}. This is a direct consequence of \cite[Theorem 3.19]{HaragusIooss.2011}, see also \cite[Theorem 3.3]{Veerman.2015}. Hence, \eqref{eq:normform} does not contain a linear term of the form $\hat{\rho}\, (\bpsi_1 , 0)^t$, nor do the nonlinear functions $\Phi$ and $\Psi$ explicitly depend on $\hat{\rho}\, \psi_1$.
\end{rmk}

Substitution of the normal form expansion \eqref{eq:normform} in \eqref{eq:fullsystem_u0} yields at $O(\mu)$
\begin{equation}\label{eq:nf_mu}
 \hmbL \bpm \Phi_{01} \\ \Psi_{01} \epm = \bpm \partial_\mu \bF \\ 0 \epm_{\bF=\bF(\bu_0;0)},
\end{equation}
which by the definition of $\hmbL$ \eqref{eq:def_bbL_R} is equivalent to
\bsub
\begin{align}
 \mrL \Phi_{01} &= \partial_\mu \bF (\bu_0;0) + \Psi_{01}, \label{eq:Phi01}\\
 \mrL^\dag \Psi_{01} &= 0.
\end{align}
\esub
We see that $\Psi_{01} \in \text{ker}(\mrL^\dag)$; hence, the solvability condition of \eqref{eq:Phi01} yields
\begin{equation}
\bpm \Phi_{01} \\ \Psi_{01} \epm = \bpm \mrL^{-1} \tPi_0^\dag \partial_\mu \bF \\ -\Pi_0^\dag \partial_\mu \bF \epm_{\bF=\bF(\bu_0;0)} + \alpha_{01} \bpm \psi_0 \\ 0 \epm + \beta_{01} \bpm \psi_1 \\ 0 \epm,
\end{equation}
with $\alpha_{01}$ and $\beta_{01}$ yet to be determined. Next, we consider the equation at $O(\mu^2)$
\begin{equation}
\hmbL\bpm \Phi_{02}\\ \Psi_{02} \epm = \bpm \frac{1}{2}(\nabla_\bu^2 \bF)(\Phi_{01},\Phi_{01}) + (\partial_\mu \nabla_\bu \bF)\,\Phi_{01} \\
(\nabla_\bu^2 \bF^\dag)(\Phi_{01},\Psi_{01}) + (\partial_\mu \nabla_\bu \bF^\dag) \Psi_{01} \epm_{\bF=\bF(\bu_0;0)}.
\end{equation}
The solvability condition for the equation for $\Psi_{02}$ stipulates that
\begin{multline}
-(\nabla_\bu^2 \bF(\bu_0;0)^\dag)(\mrL^{-1} \tPi_0^\dag \partial_\mu \bF(\bu_0;0), \Pi_0^\dag \partial_\mu \bF(\bu_0;0)) -\alpha_{01}(\nabla_\bu^2 \bF(\bu_0;0)^\dag)(\psi_0, \Pi_0^\dag \partial_\mu \bF(\bu_0;0))\\ -\beta_{01}(\nabla_\bu^2 \bF(\bu_0;0)^\dag)(\psi_1, \Pi_0^\dag \partial_\mu \bF(\bu_0;0)) - (\partial_\mu \nabla_\bu \bF(\bu_0;0)^\dag) \Pi_0^\dag \partial_\mu \bF(\bu_0;0) \perp \text{ker}(\mrL),
\end{multline}
from which follows that $\beta_{01}=0$ and
\begin{equation}
\alpha_{01} \Pi_0 (\nabla_\bu^2 \bF(\bu_0;0)^\dag)(\psi_0,\psi_0^\dag) = \Pi_0 \left((\nabla_\bu^2 \bF(\bu_0;0)^\dag)(\mrL^{-1}\tPi_0^\dag \partial_\mu \bF(\bu_0;0),\psi_0^\dag) + \partial_\mu \nabla_\bu\bF(\bu_0;0)^\dag \psi_0^\dag\right).
\end{equation}
The equation at $O(\rho^2)$
\begin{equation}
\hmbL \bpm \Phi_{20} \\ \Psi_{20} \epm = \bpm \frac{1}{2} (\nabla_\bu^2 \bF)(\psi_0,\psi_0) \\ 0 \epm_{\bF=\bF(\bu_0;0)},
\end{equation}
being of the same qualitative form as \eqref{eq:nf_mu}, can be solved to obtain
\begin{equation}
\bpm \Phi_{20} \\ \Psi_{20} \epm = \bpm \mrL^{-1} \tPi_0^\dag \frac{1}{2}(\nabla_\bu^2 \bF)(\psi_0,\psi_0) \\ -\Pi_0^\dag \frac{1}{2}(\nabla_\bu^2 \bF)(\psi_0,\psi_0) \epm_{\bF=\bF(\bu_0;0)} + \alpha_{20} \bpm \psi_0 \\ 0 \epm + \beta_{20} \bpm \psi_1 \\ 0 \epm.
\end{equation}
However, the equation at $O(\rho \mu)$
\begin{equation}
 \hmbL \bpm \Phi_{11} \\ \Psi_{11} \epm = \bpm (\nabla_\bu^2 \bF)(\psi_0,\Phi_{01}) + (\partial_\mu \nabla_\bu \bF) \psi_0 \\ (\nabla_\bu^2 \bF^\dag)(\psi_0,\Psi_{01}) \epm_{\bF=\bF(\bu_0;0)}
\end{equation}
yields as solvability condition for $\Psi_{11}$
\begin{equation}\label{eq:solcon_psi11}
(\nabla_\bu^2 \bF(\bu_0;0)^\dag)(\psi_0,\Psi_{01}) = -(\nabla_\bu^2 \bF(\bu_0;0)^\dag)(\psi_0,\Pi_0^\dag \partial_\mu \bF(\bu_0;0)) \perp \text{ker}(\mrL),
\end{equation}
which is in general \emph{not} satisfied. At the next order, we encounter a similar situation at $O(\rho^3)$, where the equation
\begin{equation}\label{eq:eq_phipsi30}
\hmbL \bpm \Phi_{30} \\ \Psi_{30} \epm = \bpm (\nabla_\bu^2 \bF)(\psi_0,\Phi_{20}) + \frac{1}{6}(\nabla_\bu^3 \bF)(\psi_0,\psi_0,\psi_0) \\ (\nabla_\bu^2 \bF^\dag)(\psi_0,\Psi_{20}) \epm_{\bF=\bF(\bu_0;0)}
\end{equation}
yields as solvability condition for $\Psi_{30}$
\begin{equation}\label{eq:solcon_psi30}
(\nabla_\bu^2 \bF(\bu_0;0)^\dag)(\psi_0,\Psi_{20}) = -(\nabla_\bu^2 \bF(\bu_0;0)^\dag)(\psi_0,\Pi_0^\dag \frac{1}{2}(\nabla_\bu^2 \bF(\bu_0;0))(\psi_0,\psi_0)) \perp \text{ker}(\mrL),
\end{equation}
which is also in general not satisfied. Furthermore, the equations at $O(\rho\mu^2)$ and $O(\rho^2 \mu)$ explicitly depend on $\Psi_{11}$, the term that yielded the problematic solvability condition \eqref{eq:solcon_psi11}.

To resolve these issues, we assume a \emph{resonance} for the problematic equations at $O(\rho\mu)$ and $O(\rho^3)$ \cite{HaragusIooss.2011}. We take 
$p,q \in \mathbb{Z}_{\geq 0}$, $p+q \geq 1$, such that $\rho \mu = \rho^p \mu^q$; likewise, we assume that there exist $r,s \in \mathbb{Z}_{\geq 0}$, $r+s \geq 1$, such that $\rho^3 = \rho^r \mu^s$. From these assumptions, it follows that
\begin{equation}
\rho = \mu^{\frac{1}{3}},\quad \rho = \mu^{\frac{1}{2}}\quad\text{or}\quad\rho = \mu, 
\end{equation}
where we ruled out $\rho=\mu^k$ with $k >1$, by standard arguments.
The choice $\rho = \mu^{\frac{1}{3}}$ yields the same insolvable equation at  $O(\mu) = O(\rho^3)$ while the
choice $\rho=\mu^{\frac12}$ yields a transverse bifurcation with persistence of the freeway solutions for $\mu>0$.
Hence, the only relevant scaling choice to be investigated is $\rho = \mu$.

To simplify notation, we rewrite the normal form expansion \eqref{eq:normform}, \eqref{eq:nf_expansion} and set
\begin{equation}\label{eq:nf_tr_expansion}
\bpm \bu \\ \bv \epm = \sum_{i=1}^N \mu^i\bpm \Phi^\text{tr}_i \\ \Psi^\text{tr}_i \epm  + O(\mu^{N+1}).
\end{equation}
Substitution of the normal form expansion \eqref{eq:nf_tr_expansion} in \eqref{eq:fullsystem_u0} yields at $O(\mu)$ 
\begin{equation}
\hmbL \bpm \Phi^\text{tr}_1 \\ \Psi^\text{tr}_1 \epm = \bpm \partial_\mu \bF \\ 0 \epm_{\bF=\bF(\bu_0;0)},
\end{equation} 
which is equivalent to \eqref{eq:nf_mu}; hence, we obtain
\begin{equation}\label{eq:nf_tr_O1sol}
\bpm \Phi^\text{tr}_1 \\ \Psi^\text{tr}_1 \epm = \bpm \mrL^{-1} \tPi_0^\dag \partial_\mu \bF \\ -\Pi_0^\dag \partial_\mu \bF \epm_{\bF=\bF(\bu_0;0)} \!\!\!\!\! + \alpha^\text{tr}_1 \bpm \psi_0 \\ 0 \epm + \beta^\text{tr}_1 \bpm \psi_1 \\ 0 \epm,
\end{equation}
with $\alpha^\text{tr}_1$ and $\beta^\text{tr}_1$ to be determined at the next order. At $O(\mu^2)$, we find
\begin{equation}
\hmbL \bpm \Phi^\text{tr}_2 \\ \Psi^\text{tr}_2 \epm = \bpm (\partial_\mu \nabla_\bu\bF) \Phi^\text{tr}_1 + \frac{1}{2} (\nabla_\bu^2 \bF)(\Phi^\text{tr}_1,\Phi^\text{tr}_1) \\ (\partial_\mu \nabla_\bu \bF^\dag)\Psi^\text{tr}_1 + (\nabla_\bu^2 \bF^\dag)(\Phi^\text{tr}_1,\Psi^\text{tr}_1) \epm_{\bF=\bF(\bu_0;0)};
\end{equation}
the solvability condition for $\Psi^\text{tr}_2$ yields $\beta^\text{tr}_1=0$ and
\begin{multline}
\alpha^\text{tr}_1 \,\Pi_0(\nabla_\bu^2 \bF(\bu_0;0)^\dag)(\psi_0,\Pi_0^\dag \partial_\mu \bF(\bu_0;0)) = -\Pi_0(\partial_\mu \nabla_\bu \bF(\bu_0;0)^\dag)\Pi_0^\dag \partial_\mu \bF(\bu_0;0) \\ -  \Pi_0(\nabla_\bu^2 \bF(\bu_0;0)^\dag)(\mrL^{-1} \tPi_0^\dag \partial_\mu \bF(\bu_0;0),\Pi_0^\dag \partial_\mu \bF(\bu_0;0)),
\end{multline}
which fully determines $\Phi^\text{tr}_1$ \eqref{eq:nf_tr_O1sol}. Furthermore, we obtain
\begin{equation}
\bpm \Phi^\text{tr}_2 \\ \Psi^\text{tr}_2 \epm = \bpm 1 & 1 \\ 0 & 1 \epm \bpm \mrL^{-1} \left[(\partial_\mu \nabla_\bu\bF) \Phi^\text{tr}_1 + \frac{1}{2} (\nabla_\bu^2 \bF)(\Phi^\text{tr}_1,\Phi^\text{tr}_1)\right] \\ (\mrL^\dagger)^{-1} \left[ (\partial_\mu \nabla_\bu \bF^\dag)\Psi^\text{tr}_1 + (\nabla_\bu^2 \bF^\dag)(\Phi^\text{tr}_1,\Psi^\text{tr}_1)\right] \epm_{\bF=\bF(\bu_0;0)} \!\!\!\!\! + \alpha^\text{tr}_2 \bpm \psi_0 \\ 0 \epm + \beta^\text{tr}_2 \bpm \psi_1 \\ 0 \epm,
\end{equation}
with $\alpha^\text{tr}_2$ and $\beta^\text{tr}_2$ to be determined at the next order. This expansion allows us to formulate the following Theorem:

\begin{thm}\label{t:tr_exist}
 Let $0<\delta\ll1$ be sufficiently small. Assume that there exists $\mu_0 > 0$ such that the freeway system \eqref{e:FW} admits a pair of orbit families $\bu_{\pm}(\mu)$ connecting the same zeros $\ba_i$ and $\ba_j$ for all $0 < \mu < \mu_0$; assume that this pair of orbit families coincides and terminates at $\bu_+(0) = \bu_-(0)=\bu_0$ through a saddle-node bifurcation; assume that this saddle-node bifurcation is nondegenerate. Denote the linearization of \eqref{e:FW} at $\bu_0$ by $\mrL$ \eqref{e:Ldef}. Then, there exists an open neighbourhood $U$ of $\mu=0$ such that for all $\mu \in U$, there exists a minimizer $\bu_\mathrm{tr}(\mu)$ of the reduced free energy $\cF_1$ \eqref{eq:F1-def}, with energy value
 \begin{equation}
 \cF_1[\bu_\mathrm{tr}(\mu)] =  \frac{\mu^2}{2}\frac{\langle \partial_\mu \bF(\bu_0;0),\psi_0 \rangle^2}{\langle \psi_0^\dag,\psi_0\rangle^2}\|\psi_0^\dag\|^2 + O(\mu^3),
 \end{equation}
 with $\psi_0$ as in \eqref{e:SN_eigenvector}, and where $\psi_0^\dag$ is the unique element of $\text{ker}(\mrL^\dag)$ that is even as a function of $z$.
\end{thm}

\begin{proof}
 The local existence of $\bu_\text{tr}(\mu)$ for small $\mu$ is an immediate consequence of the normal form expansion in section \ref{s:ex:nf_exp}. The reduced free energy \eqref{eq:F1-def} can be written in terms of the norm induced by the $L^2(\R)$-inner product as $\cF_1[\bu] = \frac{1}{2}\| D^2 \bu_{zz} - \bF(\bu;\mu) \|^2 = \frac{1}{2}\| \bv \|^2$ by \eqref{eq:uv_system_uzz}. The leading order expansion of $\bv$ given in \eqref{eq:nf_tr_O1sol} yields the energy value to leading order in $\mu$.
\end{proof}

\begin{rmk} The existence of homoclinic orbits in \eqref{e:FW} as presented in \cite{DoelmanVeerman.2015} is a consequence of the transversal intersection of manifolds, which is directly equivalent to the invertibility of $\mrL$ \eqref{e:Ldef} (up to translation). This implies invertibility of $\hmbL$ \eqref{eq:def_bbL_R}, and ensures the unique local embedding of solutions of \eqref{e:FW} in the phase space of \eqref{e:TR_system}. The toll-road branch $\bu_\mathrm{tr}$ that intersects the freeway homoclinic families $\bu_{\pm}$ at $\mu=0$, exists precisely because the invertibility of $\mrL$ fails at $\mu=0$, introducing a nontrivial (even) kernel element $\psi_0$, which is the basis for the normal form expansion in Section \ref{s:ex:nf_exp}.
\end{rmk}


\subsection{Toll-road connections in the PCB model}
The bifurcation analysis of section\,\ref{s:Barriers} allows the construction of low-energy toll-road connections. This is
relevant to situations in which mass constraints prevent the formation of freeway connections. For the PCB model of
section \ref{s:PCB_model}, the results of Theorem \ref{t:tr_exist} can be applied by extending the take-off curve to depend upon
the bifurcation parameter $\mu$, that is  $\To(s)=\To(s;\mu)$. In particular we make the following assumptions.
\begin{ass}\label{a:A4}
Let $\To(s) = \To(s;\mu)$ depend on a parameter $\mu$, and let $\rPCB = \rPCB(s;\mu)$ accordingly be as in \eqref{e:rho_PCB}. There exists $\ssn\in(0,u_{1,\mathrm{max}})$ for which $\rPCB(\ssn;0) =\rho^\prime_{\rm PCB}(\ssn;0) = 0$ and 
$$\rho^{\prime\prime}_{\rm PCB}(\ssn;0) \frac{\partial \rPCB}{\partial \mu}(\ssn;0) < 0.$$ 
\end{ass}
These assumptions  guarantee the local existence of a pair of families of homoclinic orbits in the freeway system \eqref{e:FW} that terminates in a nondegenerate saddle-node bifurcation when $\mu=0$. For the PCB model \eqref{e:PCB_model}, we find
\begin{equation}
\partial_\mu \bF_\mathrm{PCB}(\bu_0;0) = \bpm - \frac{1}{3\delta} f^2(u_{\mathrm{sn},1}) u_{\mathrm{sn},2}^2 \frac{\partial \To}{\partial\mu}(u_{\mathrm{sn},1};0) \\ 0 \epm,
\end{equation}
where $\bu_\mathrm{sn} = \left(u_{\mathrm{sn},1},u_{\mathrm{sn},2}\right)^t$ is the (degenerate) homoclinic orbit at the saddle-node bifurcation. Using Theorem \ref{t:GSP_SS}, we can obtain an explicit expression for $\psi_0$ as defined in \eqref{e:SN_eigenvector}, as follows. From Assumption \ref{a:A4}, it follows that the pair of solutions $s_\pm(\mu)$ to $\rho_\mathrm{PCB}(s;\mu)=0$ can be expanded as $s_\pm(\mu) = \ssn \pm \sqrt{\mu}\,s_1 + \mathcal{O}(\mu)$, with
\begin{equation}
s_1 := \left| \frac{\sqrt{2} \, \frac{\partial \rPCB}{\partial \mu}}{\sqrt{-\rho^{\prime\prime}_{\rm PCB} \frac{\partial \rPCB}{\partial \mu}}}\right|_{(s;\mu) = (\ssn;0)} = \left| \frac{\sqrt{2} \,\To \frac{\partial \To}{\partial \mu}}{\sqrt{\To \frac{\partial \To}{\partial \mu}\left(W''- \To'^2 - \To \To''\right)}}\right|_{(s;\mu) = (\ssn;0)}.
\end{equation}
Moreover, writing $\hat{u}(z) := u_{1,s}^\mathrm{s}(z;\ssn)$ (for the definition of $u_{1,s}^\mathrm{s}$, see Theorem \ref{t:GSP_SS}), we see that there exists a shift $z_1 < 0$ such that $u_{s,1}^\mathrm{s}(z;s_\pm(\mu)) = \hat{u}(z \pm \sqrt{\mu}\,z_1+ \mathcal{O}(\mu))$; a direct calculation shows that $z_1 = s_1 / \hat{u}'(0)$. Hence, the saddle-node eigenvector $\psi_0$ has, by Theorem \ref{t:GSP_SS}, the following leading order structure:
\begin{equation}\label{e:SN_eigenvector_PCB}
\psi_0 = \begin{cases} 
			\left(1,- \frac{f'(\ssn)}{f(\ssn)} u_{2,h}(z/\delta;\ssn)\right)^t & \text{if }  0\leq z < \sqrt{\delta},\\
			\left(\frac{1}{\hat{u}'(0)} \hat{u}'(z) , 0\right)^t & \text{if } \sqrt{\delta} \leq z,
		\end{cases}
\end{equation}
where $\psi_0$ has been scaled by $s_1$ compared to its original definition \eqref{e:SN_eigenvector}. We now use \eqref{e:SN_eigenvector_PCB} to calculate
\begin{align}
\langle \partial_\mu \bF(\bu_0;0),\psi_0 \rangle &= - \frac{1}{3\delta} \int_\R \!f^2(u_{\mathrm{sn},1}) u_{\mathrm{sn},2}^2 \frac{\partial \To}{\partial\mu}(u_{\mathrm{sn},1};0) \left(\psi_0\right)_1 \text{d} z \nonumber\\
&= - \frac{1}{3} f^2(\ssn) \frac{\partial \To}{\partial\mu}(\ssn;0) \int_\R \! u_{2,h}(\zeta;\ssn)^2  \text{d} \zeta + O(\delta),\\
&= -2 \frac{\partial \To}{\partial\mu}(\ssn;0) + O(\delta)
\end{align}
by \eqref{e:PCB_u2h}. Furthermore, we know that $\psi_0^\dag $ is the unique even element of $\text{ker } \mrL^\dag$, which therefore solves the system
\bsub\label{e:PCB_Lsystem_full}
\begin{align}
\left[\partial_z^2  - W''(u_{\mathrm{sn},1}) + \frac{1}{3 \delta}f(u_{\mathrm{sn},1})^2\left(\To'(u_{\mathrm{sn},1};0) + 2\frac{ f'(u_{\mathrm{sn},1})}{f(u_{\mathrm{sn},1})} \To(u_{\mathrm{sn},1};0)\right) u_{\mathrm{sn},2}^2 \right]\left(\psi_0^\dagger\right)_1 & \nonumber\\
 + f'(u_{\mathrm{sn},1})u_{\mathrm{sn},2}^2\left(\psi_0^\dagger\right)_2&= 0, \\
\frac{2}{3 \delta}f(u_{\mathrm{sn},1})^2 \To(u_{\mathrm{sn},1};0) u_{\mathrm{sn},2} \left(\psi_0^\dagger\right)_1 + \left[\delta^2 \partial_z^2 - 1 + 2 f(u_{\mathrm{sn},1})u_{\mathrm{sn},2}\right]\left(\psi_0^\dagger\right)_2 &= 0.
\end{align}
\esub
This system can be significantly simplified using the scale separated structure of the underlying homoclinic $\bu_\mathrm{sn}$ as given in Theorem \ref{t:GSP_SS}. In particular, outside the symmetric interval $I_f := (-\sqrt{\delta},\sqrt{\delta})$, system \eqref{e:PCB_Lsystem_full} reduces to
\bsub\label{e:PCB_Lsystem_slow}
\begin{align}
\left[\partial_z^2  - W''(u_{\mathrm{sn},1})\right]\left(\psi_0^\dagger\right)_1 &= 0, \\
\left(\psi_0^\dagger\right)_2 &= 0,
\end{align}
\esub
up to $\delta$-exponentially small terms. We note that $\left(\psi_0^\dagger\right)_1$ must be a multiple of $\hat{u}'(z)$, and fix $\left(\psi_0^\dagger\right)_1 = \frac{1}{\hat{u}'(0)} \hat{u}'(z) = \left(\psi_0\right)_1$ without loss of generality. Inside $I_f$, we rescale $\zeta = z/\delta$ and find
\bsub\label{e:PCB_Lsystem_fast}
\begin{align}
\left[\partial_\zeta^2  - \delta^2 W''(\ssn) + \frac{\delta}{3}f(\ssn)^2\left(\To'(\ssn;0) + 2\frac{f'(\ssn)}{f(\ssn)} \To(\ssn;0)\right) u_{2,h}(\zeta;\ssn)^2 \right]\left(\psi_0^\dagger\right)_1 & \nonumber\\
+ \delta^2 f'(\ssn) u_{2,h}(\zeta;\ssn)^2 \left(\psi_0^\dagger\right)_2 &= 0, \label{e:PCB_Lsystem_fast_psi1} \\
\frac{2}{3 \delta}f(\ssn)^2 \To(\ssn;0) u_{2,h}(\zeta;\ssn) \left(\psi_0^\dagger\right)_1 + \left[\partial_\zeta^2 - 1 + 2 f(\ssn,1)u_{2,h}(\zeta;\ssn)\right]\left(\psi_0^\dagger\right)_2 &= 0. \label{e:PCB_Lsystem_fast_psi2}
\end{align}
\esub
From \eqref{e:PCB_Lsystem_fast_psi2}, we infer that $\left(\psi_0^\dagger\right)_2$ scales with $1/\delta$. For the first component $\left(\psi_0^\dagger\right)_1$, this yields $\partial_\zeta^2 \left(\psi_0^\dagger\right)_1 = O(\delta)$ from which we conclude $\left(\psi_0^\dagger\right)_1 = 1$ by continuity. Rescaling the second component $\left(\hat{\psi}_0^\dagger\right)_2:=\delta \left(\psi_0^\dagger\right)_2$, it obeys
\begin{equation}\label{e:PCB_Lsystem_fast_red}
\left[\partial_\zeta^2 - 1 + 2 f(\ssn)u_{2,h}(\zeta;\ssn)\right]\left(\hat{\psi}_0^\dagger\right)_2 = - \frac{2}{3 }f(\ssn)^2 \To(\ssn;0) u_{2,h}(\zeta;\ssn).
\end{equation}
Using \eqref{e:PCB_u2h}, we can reduce \eqref{e:PCB_Lsystem_fast_red} to
\begin{equation}\label{e:PCB_Lsystem_fast_red_ex}
\left[\partial_\zeta^2 - 1 + 3 \, \text{sech}^2 (\zeta/2)\right]\left(\hat{\psi}_0^\dagger\right)_2 = - f(\ssn) \To(\ssn;0) \, \text{sech}^2 (\zeta/2),
\end{equation}
which can be solved explicitly, yielding
\begin{equation}\label{e:PCB_psi02_explicit}
\left(\hat{\psi}_0^\dagger\right)_2 =  - f(\ssn) \To(\ssn;0) \, \text{sech}^2 (\zeta/2) \left(1- (\zeta/2) \text{tanh }(\zeta/2)\right).
\end{equation}
To summarize, we have found to leading order in $\delta$
\begin{equation}\label{e:SN_adj_eigenvector_PCB}
\psi_0^\dag = \begin{cases} 
\left(1,-\frac{1}{\delta} f(\ssn) \To(\ssn;0) \, \text{sech}^2 (\zeta/2) \left(1- (\zeta/2) \text{tanh }(\zeta/2)\right)\right)^t & \text{if }  0\leq z < \sqrt{\delta},\\
\left(\frac{1}{\hat{u}'(0)} \hat{u}'(z) , 0\right)^t & \text{if } \sqrt{\delta} \leq z.
\end{cases}
\end{equation}
This allows us to calculate
\begin{align}
\| \psi_0^\dag \|^2 
&= \frac{1}{\delta} \int_\R \!\! \left(\hat{\psi}_0^\dagger\right)_2^2 \text{d} \zeta + O(1)
= \frac{1}{\delta} f(\ssn)^2 \To(\ssn;0)^2 \left(\frac{4}{3} + \frac{2\pi^2}{45}\right) + O(1)
\end{align}
and
\begin{align}
\langle \psi_0^\dag, \psi_0 \rangle &= \int_\R \left(\psi_0\right)_1 \left(\psi_0^\dag\right)_1 \text{d}z + \frac{1}{\delta} \int_\R \left(\psi_0\right)_2 \left(\hat{\psi}_0^\dag\right)_2\text{d}z\\
&= 2 \int_0^\infty \!\!\!\!\!\frac{1}{\hat{u}(0)^2} \hat{u}(z)^2\text{d} z + \frac{3}{2} \frac{f'(\ssn)}{f(\ssn)} \To(\ssn;0) \int_\R \!\!\text{sech}^4(\zeta/2)\left(1-(\zeta/2)\text{tanh }(\zeta/2)\right)\text{d}\zeta\nonumber\\
&= \frac{1}{W(\ssn)} \int_0^{\ssn}  \!\!\!\!\!\sqrt{2 W(\hat{u})}\,\text{d}\hat{u} + 3\frac{f'(\ssn)}{f(\ssn)} \To(\ssn;0)
\end{align}
to leading order in $\delta$. Using the results obtained so far, we calculate the value of the reduced free energy of the toll-road branch in the PCB model:
\begin{equation}
\label{e:TR-PCB-energy}
\cF_1[\bu_\mathrm{tr}(\mu)] = \frac{1}{\delta}\frac{\mu^2}{2}\left[\cF_1^0(\ssn) + O\left(\delta,\mu\right) \right],
\end{equation}
with
\begin{equation}
\cF_1^0(\ssn) = \left(\frac{2 f(\ssn) \To(\ssn;0) \frac{\partial \To}{\partial\mu}(\ssn;0)}{\frac{1}{W(\ssn)} \int_0^{\ssn}  \!\!\!\!\!\sqrt{2 W(\hat{u})}\,\text{d}\hat{u} + 3\frac{f'(\ssn)}{f(\ssn)} \To(\ssn;0)}\right)^2 \left(\frac{4}{3} + \frac{2\pi^2}{45}\right).
\end{equation}
For a PCB model with a prescribed take-off curve, embedding the take-off curve in a larger familty $\To(s,\mu)$ which has a saddle-node
bifurcation at $\mu=0$ and reverts to the original take-off curve at $\mu=\mu_*$, provides for the existence of a toll-road connection
with cholesterol mass scaled by $f(s_{\rm sn})$ with energy given by (\ref{e:TR-PCB-energy}) with $\mu=\mu_*.$ This relates the distance of
the take-off curve to the unstable slow manifold to the existence and energy of an associated toll-road connection.


\bibliography{Toll-Free_references}
\bibliographystyle{test4}

\end{document}